\documentclass[11pt]{amsart}
\usepackage{calc,amssymb,amsmath,ulem}
\usepackage{alltt}
\RequirePackage[dvipsnames,usenames]{color}

\input{kmacros3.sty}
\input{xy}
\input{chancery.sty}
\xyoption{all}
\usepackage[hidelinks]{hyperref}
\usepackage{xcolor}
\hypersetup{
    colorlinks,
    linkcolor={blue!50!black},
    citecolor={green!50!black},
    urlcolor={red!80!black}
}
\usepackage{graphicx}
\usepackage[all,cmtip]{xy}
\usepackage{verbatim,xspace}
\usepackage[left=1.25in,top=1.09in,right=1.25in,bottom=1.09in]{geometry}
\usepackage{mabliautoref}

\def\todo#1{\textcolor{Mahogany}%
{\sffamily\footnotesize\newline{\color{Mahogany}\fbox{\parbox{\textwidth-15pt}{#1}}}\newline}}

\renewcommand{\O}{\sO}

\renewcommand{\to}[1][]{\xrightarrow{\ #1\ }}

\newcommand{\into}[1][]{\lhook \joinrel \xrightarrow{\ #1\ }}
\newcommand{\DBDual}[1]{{{\underline \omega}{}^{\mydot}_{#1}}}

\newcommand{\stwo}{${\mathrm{S}_2}$}
\newcommand{\rone}{${\mathrm{R}_1}$}
\newcommand{\sk}[1]{${\mathrm{S}_{#1}}$}
\newcommand{\DB}{Du~Bois\xspace}
\renewcommand{\sHom}[0]{{\mcH\mco\mcm}}
\renewcommand{\sExt}[0]{{\mcE\mcx\mct}}
\renewcommand{\myR}{\mcR}
\def\frm{\mathfrak{m}}

\begin{document}
\numberwithin{equation}{theorem}
\title[Inversion of adjunction for rational and Du~Bois pairs]{Inversion of
  adjunction for rational and Du~Bois pairs}

\thanks{The first named author was supported in part by NSF Grant DMS-0856185 and
  DMS-1301888 and the Craig McKibben and Sarah Merner Endowed Professorship in
  Mathematics at the University of Washington.}

\thanks{The second named author was supported in part by the NSF grant DMS \#1064485,
  NSF FRG Grant DMS \#1501115, NSF CAREER Grant DMS \#1501102 and a Sloan
  Fellowship.}

\begin{abstract}
  We prove several results about the behavior of Du Bois singularities and Du Bois
  pairs in families.  Some of these generalize existing statements about Du Bois
  singularities to the pair setting while others are new even in the non-pair
  setting.  We also prove a new inversion of adjunction result for Du Bois and
  rational pairs.  In the non-pair setting this asserts that if a family over a
  smooth base has a special fiber $X_0$ with Du~Bois singularities and the general
  fiber has rational singularities, then the total space has rational singularities
  near $X_0$.
\end{abstract}

\author[S.~J~Kov\'acs]{S\'andor J Kov\'acs}
\address{SJK: Department of Mathematics, University of Washington,
Box 354350, Seattle, Washington, 98195 USA}
\email{skovacs@uw.edu}

\author[K.~Schwede]{Karl~Schwede}
\address{KS: Department of Mathematics, The University of Utah,
155 S 1400 E Room 233, Salt Lake City, UT 84112, USA}
\email{schwede@math.utah.edu}

\subjclass[2010]{14F18, 14B07, 14D05, 14E15, 14J17, 14B05}

\keywords{Du Bois singularities, rational singularities, inversion of adjunction,
  vanishing theorem}

\maketitle

\section{Introduction}

Rational singularities have been the gold-standard for ``mild'' singularities in
algebraic geometry for several decades.  Whenever a new class of varieties with
singularities is discovered, the first question usually asked is whether or not the
new varieties have rational singularities.  A key reason for this is that varieties
with rational singularities behave cohomologically as if they were smooth.  However,
for many purposes rational singularities are not broad enough. For instance, nodes
are not rational singularities, and more generally, singularities appearing on stable
varieties, that is, mild degenerations of smooth ones that are necessary to consider
in order to compactify moduli spaces are not always rational.  The class of Du Bois
(or DB) singularities is slightly more inclusive than rational singularities.  \DB
singularities behave cohomologically as if they had simple normal crossing
singularities (i.e., a higher dimensional version of nodes).

Recently, Koll\'ar and the first named author
\cite{KovacsDBPairsAndVanishing,KollarKovacsSingularitiesBook} have introduced the
notions of rational and Du Bois pairs $(X, D)$ for a normal variety $X$ and a reduced
divisor $D \subseteq X$.  These notions are philosophically distinct from the
singularities considered typically in the minimal model program since $(X, D)$ having
rational (respectively Du Bois) singularities does not generally imply that the
ambient space $X$ has rational (respectively Du Bois) singularities \cite[Remark
2.81(2)]{KollarKovacsSingularitiesBook} (respectively \autoref{ex.DBPairButNotDB},
\autoref{ex.NormalDBPairButNotDB}, \cite{GK13}).  Instead the singularities of $(X,
D)$ measure the connection between the singularities of $X$ and $D$ (a notion
obviously connected with problems related to inversion of adjunction).  Furthermore,
like \DB singularities, if $(X, Z)$ is a \DB pair then the ideal sheaf of $Z$
satisfies various Kodaira-type vanishing theorems, an observation which we hope will
be useful in the future.

Even though a priori rational and \DB singularities are not part of the class one
usually associates with the minimal model program (mmp), these singularities play important
roles in both the mmp and moduli theory via the fact that (semi-)log canonical
singularities are \DB
\cite{KovacsSchwedeSmithLCImpliesDuBois,KollarKovacsLCImpliesDB}.  In addition, \DB
singularities have also played important roles in various other contexts recently.
They are arguably the largest class of singularities for which we know that Kodaira
vanishing holds \cite{PatakfalviSemiNegativityHodgeBundles}, they appear in proofs of
extension and other vanishing theorems \cite{MR2854859}, positivity theorems
\cite{MR2969273}, categorical resolutions \cite{MR2981713}, log canonical
compactifications \cite{MR3032329} and many other results more directly related to
the minimal model program.

It is now a basic tenet of the minimal model program that the right way to study
singularities is via pairs cf.\
\cite{KollarSingularitiesOfPairs,KollarKovacsSingularitiesBook}. This allows for more
freedom in applications and makes inductive arguments easier. The same is true for
rational and \DB singularities. The introduction of \DB pairs streamlined some
existing proofs cf.\ \cite[Chapter~6]{KollarKovacsSingularitiesBook} and extended the
realm of applications.

In this paper we extend several recent results on \DB singularities to the context of
\DB pairs, notably the recent results on deformations of \DB singularities found in
\cite{KovacsSchwedeDBDeforms} and the requisite injectivity theorem, a result of
Koll\'ar and Kov\'acs on the behavior of depth in \DB families, and the
characterization of Cohen-Macaulay \DB singularities of
\cite{KovacsSchwedeSmithLCImpliesDuBois}.

Furthermore, we prove a new inversion of adjunction statement for rational and \DB
pairs. This statement is new even in the non-pair setting.  Roughly speaking, in the
non-pair setting it says that if $f : X \to B$ is a family over a smooth base such
that the general fiber has rational singularities and the special fiber has \DB
singularities, then $X$ has rational singularities in a neighbourhood of the special
fiber.  See Theorem~E below for the general statement.

We state each of these theorems below.  We begin with the deformation statement.

\begin{theoremA*}[\autoref{thm.MainTheoremDeformation}]
  Let $X$ be a reduced scheme essentially of finite type over $\bC$, $Z \subseteq X$ is a
  reduced subscheme and $H$ a reduced effective Cartier divisor on $X$ that does not
  contain any component of $Z$.  If $(H, Z\cap H)$ is a \DB pair, then $(X,Z)$ is a
  \DB pair near $H$.
\end{theoremA*}

Just as in the non-pair setting, to prove this we first show an injectivity theorem.

\begin{theoremB*}[\autoref{thm.MainInjectivityForPairs}]
  Let $X$ be a reduced scheme over $\bC$ and $Z \subseteq X$ a reduced subscheme.
  Then the natural map
  \[
  \Phi^j : \sExt^j_{\O_X}(\DuBois{X,Z}, \omega_X^{\mydot}) \into
  \sExt^j_{\O_X}(\sI_Z, \omega_X^{\mydot})
  \]
  is injective for every $j\in\bZ$.
\end{theoremB*}
Here $\sExt^j_{\O_X}(\blank, \omega_X^{\mydot})$ is shorthand to denote $\myH^j(\myR
\sHom_{\O_X}(\blank, \omega_X^{\mydot}))$.

We also generalize some of the results of \cite{KollarKovacsLCImpliesDB} for
families to the context of \DB pairs.

\begin{theoremC*}[\autoref{cor.CMClosedDuBoisFamilies}]
  Let $f: (X, Z) \to B$ be a flat projective family with $\O_Z, \sI_Z$ flat over
  $B$ as well.  Assume that all the fiber pairs $(X_b, Z_b)$ are \DB.  Assume also that $V$ is connected and the generic fibers $(\sI_Z)_{\gen}$
  are Cohen-Macaulay, then all the fibers $(\sI_Z)_b$ are Cohen-Macaulay.
\end{theoremC*}

We have a multiplier ideal/module like characterization of \DB pairs.

\begin{theoremD*}[\autoref{thm.KSSPairs}]
  Let $X$ be a normal variety and $Z \subseteq X$ a divisor. Further let $\pi : \tld
  X \to X$ be a log resolution of $(X, Z)$ with $E = \pi^{-1}(Z)_{\red} \vee
  \exc(\pi)$.  If $\sI_Z$ is Cohen-Macaulay then $(X, Z)$ is Du Bois if and only if
  \[
  \pi_* \omega_{\tld X}(E) \simeq \omega_X(Z).
  \]
\end{theoremD*}

All of the results above are used in the proof of our inversion of adjunction result.

\begin{theoremE*}[\autoref{thm.InvOfAdjForDuBoisPairs}]\label{thmE}
  Let $f : X \to B$ be a flat projective geometrically integral family over a smooth connected base $B$ with $\dim B \geq 1$, $H =
  f^{-1}(0)$ the special fiber, and $D$ a reduced codimension $1$ subscheme of $X$
  which is flat over $B$.
  Assume that $(H, D|_H)$ is a \DB pair and that $(X \setminus H, D \setminus H)$ is
  a rational pair.  Then $(X, D)$ is a rational pair.
\end{theoremE*}

This last result is \emph{new} even in the case $D = 0$, see
\autoref{cor.DBRationalDeform}.  In the special case when $X \setminus H$ is smooth
and $D = 0$, Theorem E follows from \cite[Theorem 5.1]{SchwedeEasyCharacterization}.

Statements similar to Theorem E have been proved in many related situations.  For
instance, assume $D = 0$, $X \setminus H$ is canonical and $H$ is semi-log canonical,
then it follows from inversion of adjunction that $X$ has canonical singularities,
see \cite[Theorem 5.1]{KollarShepherdBarron}, \cite[Theorem 2.5]{KaruBoundedness},
\cite{KawakitaInversion}.  Here is a non-exhaustive list of some of other related
results: \cite[Proposition 2.13]{FedderWatanabe}, \cite{SchwedeFAdjunction},
\cite{HaconLCInversionAdjunction} and \cite{Lindsay14}.

\vskip 9pt

\noindent {\sf Acknowledgements:} The authors thank Zsolt Patakfalvi for several
useful discussions and the referee for numerous helpful comments and corrections.

\section{Definitions and basic properties}

\subsection{Rational pairs}

First we recall the notion of rational pairs defined by the first named author and
Koll\'ar as described in \cite[Chapter 2]{KollarKovacsSingularitiesBook}.  Note that
a similar notion was defined by the second named author and Takagi in
\cite{SchwedeTakagiRationalPairs}. The two notions are closely related but different.
Their relationship is similar to how dlt singularities compare to klt singularities.
Here we will discuss the former notion which in the dlt vs klt analogy corresponds to
dlt.


In this subsection we work over an algebraically closed field $k$ although in the
rest of the paper we restrict to working over the complex numbers.

\begin{definition}
  Let $X$ be a normal variety and $D \subseteq X$ an integral Weil divisor on $X$.  A
  \emph{log resolution} $\xymatrix{(Y, D_Y) \ar[r]^\pi & (X, D)}$ is a resolution of
  singularities such that $D_Y$ is the strict transform of $D$, and such that
  $(D_Y)_{\reduced} \cup \exc(\pi)$ is a simple normal crossing divisor.
\end{definition}

\begin{definition}
  A \emph{reduced pair} $(X,D)$ consists of a normal variety $X$ and a reduced
  divisor $D$ on $X$. For the definition of an \emph{snc pair}, the \emph{strata} of
  an snc pair and other normal crossing conditions please refer to
  \cite[Definition 1.7]{KollarKovacsSingularitiesBook}.
\end{definition}

One frequently wants log resolutions that do not blowup unnecessary centers.  One
good way to achieve this is with a thrifty resolution.

\begin{definition}[Thrifty resolution
  \protect{\cite[Definition 2.79]{KollarKovacsSingularitiesBook}}]
  \label{def.ThriftyResolution}
  Let $(X,D)$ be a reduced pair.  A \emph{thrifty resolution} of $(X, D)$ is a
  resolution $\pi : Y \to X$ such that:
  \begin{itemize}
  \item[(a)] $D_Y = \pi^{-1}_* D$ is a simple normal crossing divisor:
  \item[(b)] $\pi$ is an isomorphism over the generic point of every stratum of the
    snc locus of $(X, D)$ and $\pi$ is an isomorphism at the generic point of every
    stratum of $(Y_, D_Y)$.
  \end{itemize}
  Item (b) can also be replaced by:
  \begin{itemize}
  \item[(b$'$)] The exceptional set $E$ of $\pi$ does not contain any stratum of $(Y,
    D_Y)$ and $\pi(E)$ does not contain any stratum of the simple normal crossing
    locus of $(X, D)$.
  \end{itemize}
\end{definition}

  We can now define rational pairs \cite[Section 2.5]{KollarKovacsSingularitiesBook}.

  \begin{definition}[Rational pairs]
  \label{def.RationalPairs}
  A reduced pair $(X, D)$ is called a \emph{rational pair} if there exists a thrifty
  resolution $\pi : (Y, D_Y) \to (X, D)$ such that
  \begin{itemize}
  \item[(i)]  $\O_X(-D) \simeq \pi_* \O_Y(-D_Y)$.
  \item[(ii)]  $\myR^i \pi_* \O_Y(-D_Y) = 0$ for all $i > 0$.
  \item[(iii)]  $\myR^i \pi_* \omega_Y(D_Y) = 0$ for all $i > 0$.
  \end{itemize}
  \end{definition}

  If $(X, D)$ is a rational pair, and is in characteristic zero, then every thrifty
  resolution satisfies the properties (i), (ii), (iii) above \cite[Corollary
  2.86]{KollarKovacsSingularitiesBook}.  Even better though, property (iii) always
  holds in characteristic zero as we point out below, whether or not $(X, D)$ is a
  rational pair.
  \begin{theorem}
  \label{thm.GRVanishingForThrifty}
  Assume that $\mathrm{char}\, k=0$ and let $(X, D)$ be a reduced pair and $\pi : (Y,
  D_Y) \to (X, D)$ a thrifty resolution. Then $\myR^i \pi_* \omega_Y(D_Y) = 0$ for
  all $i > 0$.
  \end{theorem}
  \begin{proof}
    This follows from \cite[Theorem 10.39]{KollarKovacsSingularitiesBook}.
    Alternatively, one can prove this directly:
  \begin{claim}\label{iteratedGRV}
    Let $\pi : E \to D$ be a proper birational map between reduced equidimensional
    $\bC$-schemes of finite type such that $E$ is a simple normal crossing divisor in
    some smooth ambient space.  Assume that $\pi$ is birational onto its image when
    restricted to every strata of $E$ (in particular, also each irreducible component
    of $E$).  Then $\myR^i \pi_* \omega_E = 0$ for $i > 0$.
  \end{claim}
  \begin{proof}[Proof of \autoref{iteratedGRV}]
    We proceed by induction on $\dim D$ and the number of irreducible components of
    $E$, and note that the base case is simply Grauert-Riemenschneider vanishing
    \cite{GRVanishing}.  Write $E = E_0 \cup E'$ where $E_0$ is an irreducible
    component of $E$ and $E'$ are the remaining irreducible components.  We have a
    short exact sequence
  \[
  0 \to \O_E \to \O_{E_0} \oplus \O_{E'} \to \O_{E_0 \cap E'} \to 0.
  \]
  Dualizing we obtain
  \[
  0 \to \omega_{E_0} \oplus \omega_{E'} \to \omega_{E} \to \omega_{E_0 \cap E'} \to
  0.
  \]
  Notice that $E_0 \cap E'$ is a simple normal crossings divisor in the smooth
  ambient space $E_0$.  $E_0 \cap E'$ is also a union of strata of $E$ and hence
  $\pi$ is still birational when restricted to each strata of $E' \cap E_0$.
  Applying $\myR^i \pi_*$ and the inductive hypothesis to $E_0$, $E'$ and $E_0 \cap
  E'$ proves the claim.
\end{proof}
  \noindent {\it Alternative Proof of \autoref{thm.GRVanishingForThrifty}.}
  Push forward the short exact sequence $0 \to \omega_Y \to \omega_Y(D_Y) \to
  \omega_{D_Y} \to 0$ via $\pi$ and apply the claim to $\omega_{D_Y}$ and
  $\omega_Y$.  Note that the thrifty resolution hypothesis guarantees that $\pi$ is
  birational when restricted to any strata of $D_Y$ (since it isomorphism at the
  generic point of each strata).
\end{proof}

\noindent
This gives us the following criterion.

\begin{proposition}\label{cor:rational-pairs-criterion}
  Let $(X, D)$ be a reduced pair and $\pi : (Y, D_Y) \to (X, D)$ a thrifty
  resolution. Then $(X, D)$ is a rational pair if and only if
  \[
  \myR \sHom_{\O_X}^{\mydot}(\O_X(-D), \omega_X^{\mydot}) \qis \myR \pi_*
  \omega_Y(D_Y)[\dim X] \qis \pi_* \omega_Y(D_Y)[\dim X]
  \]
  for some thrifty resolution. Furthermore, in characteristic zero the second
  isomorphism is automatic.
\end{proposition}

\begin{proof}
  Observe that the conditions (i) and (ii) of \autoref{def.RationalPairs} are
  equivalent to the isomorphism $\myR \pi_* \O_Y(-D_Y) \qis \O_X(-D)$.  Applying
  Grothendieck duality and condition (iii) to this isomorphism yields the statement.
  The characteristic zero statement is simply \autoref{thm.GRVanishingForThrifty}.
\end{proof}

\subsection{Notation}
\noindent
Throughout the rest of this paper, all schemes will be assumed to be Noetherian
separated schemes and essentially\footnote{that is, a localization of a finite type
  scheme} of finite type over $\bC$.  Given divisors $D = \sum a_i D_i$ and $D' =
\sum b_i D_i$ on a normal variety (possibly allowing $a_i, b_j$ to be zero), we
define $D \vee D' = \sum \max(a_i, b_i) D_i$ and $D \wedge D' = \sum \min(a_i, b_i)
D_i$.  Of course, if $D$ and $D'$ have no common components then $D \vee D' = D+D'$
and $D \wedge D' = 0$.  On a scheme $X$ essentially of finite type over $\bC$, we use
$\myD(\blank) = \myR \sHom_{\O_X}^{\mydot}(\blank, \omega_X^{\mydot})$ to denote the
Grothendieck duality functor.

\subsection{Du Bois pairs}

The notion of Du~Bois singularities is becoming more and more part of the basic
knowledge in higher dimensional geometry. In particular, for the notion of the
Deligne-Du~Bois complex of a scheme of finite type over $\bC$ and its degree zero
associated graded complex, denoted by $\DuBois{X}$, we refer the reader to
\cite[Section 6.1]{KollarKovacsSingularitiesBook}.

In contrast, the notion of \DB pairs is relatively new and so here we discuss some of its basic properties.

Given a (possibly non-reduced) subscheme $Z \subseteq X$ one has an induced map in
$D^b_{\coherent}(X)$,
\[
\DuBois{X} \to \DuBois{Z},
\]
noting that by definition $\DuBois{Z} = \DuBois{Z_{\reduced}}$.  Then $\DuBois{X, Z}$
is defined to be the object in the derived category making the following an exact
triangle:
\begin{equation}
  \label{eq:1}
  \DuBois{X, Z} \to \DuBois{X} \to \DuBois{Z} \xrightarrow{+1}.
\end{equation}
If $\sI_Z$ is the ideal sheaf of $Z$, then it is easy to see that there is a natural
map $\sI_Z \to \DuBois{X,Z}$, \cite[Section 3.D]{KovacsDBPairsAndVanishing}.

\begin{definition}\cite[Definition 3.13]{KovacsDBPairsAndVanishing}
  We say that $(X, Z)$ is a \emph{Du Bois pair} (or simply a \emph{DB pair}) if the
  above map $\sI_Z \to \DuBois{X,Z}$ is a quasi-isomorphism.
\end{definition}

In the original definition of a \DB pair in \cite{KovacsDBPairsAndVanishing} it was
assumed that $Z$ is reduced.  As it turns out, this is not a necessary hypothesis.

\begin{lemma}\label{lem.DBPairImpliesReduced}
  If $(X, Z)$ is a \DB pair and  $X$ is reduced, then $Z$ is reduced.
\end{lemma}
\begin{proof}
  Note that $\DuBois{Z} = \DuBois{Z_{\reduced}}$ and so $\DuBois{X, Z} \qis
  \DuBois{X, Z_{\reduced}}$.  On the other hand, we also have an exact sequence:
  \[
  \xymatrix{ 0 \ar[r] & \myH^0(\DuBois{X, Z}) \ar@{.>}[d]_{\simeq} \ar[r] &
    \myH^0(\DuBois{X})
    \ar[d]_{\simeq} \ar[r] & \myH^0(\DuBois{Z}) \ar[d]_{\simeq} \\
    0 \ar[r] & \sI_{\mathrm{Im}\left(Z_{\reduced}^{\textnormal{SN}}\right) \subseteq
      X^{\textnormal{SN}}} \ar[r] & \O_{X^{\textnormal{SN}}} \ar[r]&
    \O_{Z_{\reduced}^{\textnormal{SN}}} }
  \]
  where $X^{\textnormal{SN}}, Z_{\reduced}^{\textnormal{SN}}$ are the
  seminormalizations of $X$ and $Z_{\reduced}$ respectively, and the right two
  isomorphisms come from \cite{SaitoMixedHodge}.  
  Note that the scheme theoretic
  image of $Z_{\reduced}^{\textnormal{SN}}$ in $X^{\textnormal{SN}}$ is reduced.  
  The fact that the left-most vertical map is an isomorphism implies that
  $\myH^0(\DuBois{X, Z})$ is a radical ideal in $\O_{X^{\textnormal{SN}}}$.  Since
  $(X, Z)$ is Du Bois, we see that $\myH^0(\DuBois{X,Z}) = \sI_{Z \subseteq X}$ and
  hence $\sI_{Z \subseteq X}$ is radical in $\O_{X^{\textnormal{SN}}}$ and hence also
  in $\O_X = \O_{X^{\reduced}}$ as desired.
\end{proof}

Frequently we will take the Grothendieck dual of $\DuBois{X,Z}$.  Hence, following
the notation of \cite{KovacsSchwedeDBDeforms} we will write
\begin{equation}
  \label{eq.DBDualDefinition}
  \DBDual{X,Z} := \myR \sHom_{\O_X}^{\mydot}(\DuBois{X,Z}, \omega_X^{\mydot}).
\end{equation}

The reader is referred to \cite[Section 6.1]{KollarKovacsSingularitiesBook} for basic
properties of \DB pairs.  As mentioned in the introduction, this notion of pairs is
somewhat different in flavor from the definition of $(X, Z)$ being log canonical or
log terminal. Being a \DB pair is more a statement about the relationship between $X$
and $Z$, not an absolute statement about the singularities of $X$ or $Z$
separately. In particular, \autoref{ex.DBPairButNotDB} and
\autoref{ex.NormalDBPairButNotDB} show that $(X, Z)$ being \DB \emph{does not} imply
that $X$ is \DB.

\begin{example}[A \DB pair whose ambient space is not \DB]
  \label{ex.DBPairButNotDB}
  Let $R$ denote the pullback of the following diagram:
  \[
  \xymatrix{
  k[x] \simeq k[x,y]/\langle y \rangle & \ar[l] k[x,y] \\
  k[x^2,x^3] \ar[u] & \ar@{.>}[l] R \ar@{.>}[u]
  }
  \]
  where the non-dotted arrows are induced in the obvious ways.  It is easy to see
  that $R = k[x^2, x^3, y, yx]$.  By construction $X = \Spec R$ is not \DB since it
  is not seminormal.  However, we claim that the pair $(\Spec R, V(\langle y, yx
  \rangle_R) )$ is \DB.  Consider the following diagram:
  \[
  \xymatrix{ 0 \ar[r] & \langle y, yx \rangle_R \ar[d]_{\alpha} \ar[r] & R \ar[r]
    \ar[d]_{\beta}& k[x^2, x^3] \ar[d]_{\gamma} \ar[r] & 0\\
    0 \ar[r] & \langle y \rangle_{k[x,y]} \ar[r] & k[x,y] \ar[r] & k[x] \ar[r] & 0.
  }
  \]
  The maps labeled $\beta$ and $\gamma$ are the seminormalizations but $\alpha$ is an
  isomorphism.  On the other hand, we know that $\DuBois{X} = \DuBois{X^{\sn}}$ in
  general since they have the same hyperresolution.  Therefore, up to harmless
  identification of modules with sheaves on an affine scheme, we see $\DuBois{\Spec
    R} \qis k[x,y]$ and $\DuBois{\Spec k[x^2, x^3]} \qis k[x]$ and so
  \[
  \langle y, yx \rangle_R = \langle y \rangle_{k[x,y]} \qis \DuBois{\Spec R,
    V(\langle y, yx \rangle_R)}.
  \]
  This proves that $(\Spec R, V(\langle y, yx \rangle_R) )$ is \DB and completes the
  example.
\end{example}

Next we will give an example of a normal \DB pair whose ambient space is not \DB. To
this end we will use a criterion for a cone being a \DB pair. In order to do that we
need to recall a definition \cite[III.3.8]{KollarKovacsSingularitiesBook}:

Let $X$ be a projective scheme and $\sL$ an ample line bundle on $X$. We will need
the \emph{spectrum of the section ring of $\sL$},
$$
C_a(X,\sL):= \Spec_k \bigoplus_{p\geq 0} H^0(X,\sL^p),
$$
which is also called the (generalized) \emph{ample cone} over $X$ with co-normal
bundle $\sL$.  If no confusion is likely, in particular when $\sL$ is fixed, we will
use the shorthand of $CX:=C_a(X,\sL)$.  Notice that for a subscheme $Z\subseteq X$
there is a natural map $\iota\!: C_a(Z, \sL|_Z) \to C_a(X, \sL)$ which is a closed
embedding away from the vertex $P \in CX$. By a slight abuse of notation we will also
use $(CX,CZ)$ to denote the pair $(C_a(X,\sL), \iota(C_a(Z,\sL|_Z)))$.

Now we are ready to state the needed \DB criterion:

\begin{proposition}\textnormal{(\cite{GK13}
    \cf\cite{MaFInjectiveBuchsbaum})}\label{prp:cone DB criterion}
  Let $X$ be a smooth projective variety, $Z \subset X$ an snc divisor (possibly the
  empty set), and $\sL$ an ample line bundle on $X$.  Then $(CX, CZ)$ is a \DB pair
  if and only if
  \[
  H^i(X, \sL^m(-Z)) = 0
  \]
  for all $i, m >0$.
\end{proposition}
\begin{proof}
  In the case $Z=\emptyset$ this follows from \cite[Theorem
  4.4]{MaFInjectiveBuchsbaum}.  The general case works similarly. For a direct proof
  see \cite[Theorem 2.5]{GK13}.
\end{proof}

While the above is sufficient for our purposes, we also obtained independently a
slightly different statement using similar methods.

\begin{lemma}[\DB pairs for graded rings]
\label{lem.DBPairsForGradedRings}
Let $X$ be a projective variety with \DB singularities, $\sL$ an ample line bundle
and $D$ a reduced connected divisor on $X$. Assume that $D$ also has only \DB
singularities.  Form the corresponding section ring $S = \bigoplus_{i \geq 0}
\Gamma(X, \sL^i)$ and $I = \bigoplus_{i \geq 0} \Gamma(X, \O_X(-D) \otimes \sL^i)$.
Fix $\bm = S_+$ to be the irrelevant ideal.  Set $Y = CX = \Spec S$ and $Z = CD =
\Spec (S/I)$.  If
\begin{equation}
\label{eq.FreeVanishingCohomology}
H^1(X, \O_X(-D) \otimes \sL^i) = 0
\end{equation} for $i \geq 0$ so that $S/I \simeq \bigoplus_{i \geq 0} \Gamma(D,
\sL^i|_D)$ then for all $i \geq 1$ we have
\[
\myH^i (\DuBois{Y,Z}) \simeq [H^{i+1}_{\bm}(I)]_{> 0}.
\]
Again under the hypothesis {\rm \autoref{eq.FreeVanishingCohomology}}, we see
immediately that $(Y, Z)$ is \DB if and only if $[H^i_{\bm}(I)]_{> 0} = 0$ for every
$i > 0$.
\end{lemma}

\begin{proof}
  First observe that both $Y$ and $Z$ are seminormal since they are saturated section
  rings over seminormal schemes.  L.~Ma showed in \cite[Equation (4.4.4) in the proof
  of Theorem 4.4]{MaFInjectiveBuchsbaum} that
\begin{equation}
\label{eq.MaIdentification}
\myH^i (\DuBois{Y}) \simeq [H^{i+1}_{\bm}(S)]_{> 0}
 \end{equation}
 for $i > 0$.  Likewise $\myH^i (\DuBois{Z}) = [H^{i+1}_{\bm}(S/I)]_{> 0}$ for $i > 0$.
 Now we analyze $\myH^{i+j} (\myR \Gamma_{\bm}(\DuBois{Y}))$ via a spectral sequence.
 Since $Y$ is \DB outside of the origin $V(\bm)$, we see that $\myH^j (\DuBois{Y})$ is
 supported only at the origin for $j > 0$.  It follows that the $E_2$-page of the
 spectral sequence
 \[
 H^i_{\bm}( \myH^{j} (\DuBois{Y})) \Rightarrow \myH^{i+j} (\myR \Gamma_{\bm}(\DuBois{Y}))
 \]
 looks like
 \[
 \xymatrix@R=12pt@C=12pt{
   \ldots  & \ldots  & \ldots  & \ldots  & \ldots & \ldots \\
   \myH^3 (\DuBois{Y}) & 0 &0 & 0 & 0 & \ldots \\
   \myH^2 (\DuBois{Y}) & 0 &0 & 0 & 0 & \ldots \\
   \myH^1 (\DuBois{Y}) \ar[drr] & 0 &0 & 0 & 0 & \ldots \\
   0 & H^1_{\bm}(S) & H^2_{\bm}(S)& H^3_{\bm}(S)& H^4_{\bm}(S) & \ldots }
\]
Here we are using the fact that $S$ is seminormal, and so $\myH^0 (\DuBois{Y}) = S$.
It is not difficult to see that the unique nonzero map of the $(i-1)$st page of this
spectral sequence induces the isomorphism of \autoref{eq.MaIdentification}, and so
those unique non-zero maps are injective.  Thus the spectral sequence contains the
data of a long exact sequence
\[
0 \to H^1_{\bm}(S) \twoheadrightarrow \bH^{1}_{\bm}(\DuBois{Y}) \to \myH^1 (\DuBois{Y})
\hookrightarrow H^2_{\bm}(S) \twoheadrightarrow \bH^{2}_{\bm}(\DuBois{Y}) \to \myH^2
(\DuBois{Y}) \hookrightarrow H^3_{\bm}(S) \twoheadrightarrow \dots
\]
Hence $\bH^{i}_{\bm}(\DuBois{Y}) = [H^i_{\bm}(S)]_{\leq 0}$ for $i \geq 2$ and
$\bH^{1}_{\bm}(\DuBois{Y}) \simeq H^1_{\bm}(S)$.  Likewise $\bH^{i}_{\bm}(\DuBois{Z})
= [H^i_{\bm}(S/I)]_{\leq 0}$ for $i \geq 2$ and $\bH^{1}_{\bm}(\DuBois{Z}) \simeq
H^1_{\bm}(S/I)$.  Furthermore, since $Y$ and $Z$ are seminormal we see that $\myH^0
(\DuBois{X,Z}) = I$ and so the same spectral sequence argument implies that we have a
long exact sequence
\[
0 \to H^1_{\bm}(I) \twoheadrightarrow \bH^{1}_{\bm}(\DuBois{Y,Z}) \to \myH^1
(\DuBois{Y,Z}) \to H^2_{\bm}(I) \twoheadrightarrow \bH^{2}_{\bm}(\DuBois{Y,Z}) \to
\myH^2 (\DuBois{Y,Z}) \to H^3_{\bm}(I) \twoheadrightarrow \dots
\]
We still have the labeled surjectivities by the Matlis dual of
\autoref{thm.MainInjectivityForPairs}, which we will prove later (we assume it for
now).  Thus it is enough to see that the maps above make the identification
$\bH^{i}_{\bm}(\DuBois{Y,Z}) = [H^i_{\bm}(I)]_{\leq 0}$ for $i \geq 2$.

We consider the diagram with distinguished triangles as rows
\[
\xymatrix{
  I \ar[d] \ar[r] & S \ar[d]\ar[r] &  S/I \ar[d]\ar[r]^-{+1} &\\
  \DuBois{Y,Z} \ar[r] & \DuBois{Y} \ar[r] & \DuBois{Z} \ar[r]^-{+1} &.  }
\]
We will apply the functor $\myR \Gamma_{\bm}(\blank)$ and take cohomology $i \geq 1$
to obtain:
\[
\xymatrix{ H^{i}_{\bm}(S) \ar[d]_{\alpha}\ar[r] & H^{i}_{\bm}(S/I) \ar[d]_{\beta}
  \ar[r] & H^{i+1}_{\bm}(I) \ar[d]_{\gamma} \ar[r] & H^{i+1}_{\bm}(S)
  \ar[d]_{\delta}\ar[r] &
  H^{i+1}_{\bm}(S/I) \ar[d]_{\epsilon}\\
  \bH^{i}_{\bm}(\DuBois{Y}) \ar[r] & \bH^{i}_{\bm}(\DuBois{Z}) \ar[r] &
  \bH^{i+1}_{\bm}(\DuBois{Y,Z}) \ar[r] & \bH^{i+1}_{\bm}(\DuBois{Y}) \ar[r] &
  \bH^{i+1}_{\bm}(\DuBois{Z})\\
  [H^{i}_{\bm}(S)]_{\leq 0} \ar@{=}[u] & [H^{i}_{\bm}(S/I)]_{\leq 0} \ar@{=}[u] & &
  [H^{i+1}_{\bm}(S)]_{\leq 0} \ar@{=}[u] & [H^{i+1}_{\bm}(S/I)]_{\leq 0} \ar@{=}[u].
}
\]
Note that $\gamma$ is the map we already identified as surjective above.  It is easy
to see that the vertical maps $\alpha, \beta, \delta, \epsilon$ are the projections
and so $[\alpha]_{\leq 0}, [\beta]_{\leq 0}, [\delta]_{\leq 0},$ and
$[\epsilon]_{\leq 0}$ are isomorphisms.  Thus $[\gamma]_{\leq 0}$ is also an
isomorphism.  But from the second row we see that $\bH^{i}_{\bm}(\DuBois{Y,Z})$ is of
non-positive degree so that $\bH^{i}_{\bm}(\DuBois{Y,Z}) = [H^i_{\bm}(I)]_{\leq 0}$
for $i \geq 2$ which completes the proof.
\end{proof}

\begin{remark}
  It would be natural to try to prove a common generalization of the (independently
  obtained) \autoref{prp:cone DB criterion} and \autoref{lem.DBPairsForGradedRings}.
\end{remark}

\begin{example}[A normal \DB pair whose ambient space is not \DB]
  \label{ex.NormalDBPairButNotDB}
  Let $W$ be an arbitrary smooth canonically polarized variety, that is, $W$ is
  smooth and projective and $\omega_W$ is ample. Further let $n > 1$, set $X=W\times \bP^n$ and $\sL=\pi_1^*\omega_W\otimes
  \pi_2^*\sO_{\bP^n}(1)$. Finally, let $K\subseteq \P^n$ be a smooth hypersurface of
  degree $n+1$, that is, $\sO_{\bP^n}(K)\simeq \omega_{\bP^n}^{-1}$, and let $Z=
  W\times K$. We claim that, using the above notation, $(CX,CZ)$ is a \DB pair, while
  $CX$ itself is not. Note also, that by construction $CX$ is normal.

  Consider $H^1(X, \sO_X(-Z) \tensor \sL^j)$ for $j \geq 0$ and observe that
  \[
  \sO_X(-Z) \tensor \sL^j = \pi_2^* \sO_{\bP^n}(-n-1) \tensor \pi_1^* \omega_W^j
  \tensor \pi_2^* \sO_{\bP^n}(j) = \pi_1^* \omega_W^j \tensor
  \pi_2^*\sO_{\bP^n}(j-n-1).
  \]
  Now $H^1(\bP^n, \sO_{\bP^n}(j-n-1)) = 0$ for all $j \geq 0$ and $H^0(\bP^n,
  \sO_{\bP^n}(j-n-1)) = 0$ for $j \leq n$.  But if $j > n \geq 1$, then $H^1(W,
  \omega_W^j) = 0$ by Kodaira vanishing and so it follows by the K\"unneth formula
  that $H^1(X, \sO_X(-Z) \tensor \sL^j) = 0$ for all $j\geq 0$, so the hypotheses of
  \autoref{lem.DBPairsForGradedRings} are satisfied.

  Let $r=\dim W$ and consider $H^r(X,\sL)$. By the K\"unneth formula
  $$
  H^r(X,\sL)\supseteq H^r(W,\omega_W)\otimes H^0(\bP^n, \sO_{\bP^n}(1))\neq 0,
  $$
  and hence by \autoref{prp:cone DB criterion} $CX$ is \emph{not} \DB.

  On the other hand  we have that
  \begin{equation}
    \label{eq:4}
    \sL(-Z) \simeq \pi_1^*\omega_W \otimes \pi_2^*\sO_{\bP^n}(1-n-1) \simeq \omega_X
    \otimes \pi_2^* \sO_{\bP^n}(1).
  \end{equation}
  Now observe that $H^q(\bP^n, \sO_{\bP^n}(1-n-1))=0$ for \emph{all}
  $q \geq 0$, so, again by the K\"unneth formula it follows that $H^i(X,\sL(-Z))=0$
  for all $i>0$.

  In order to conclude that $(CX,CZ)$ is a \DB pair we need that $H^i(X,\sL^m(-Z))=0$
  for all $i,m>0$. We just showed that $H^i(X, \sL(-Z)) = 0$ for all $i > 0$ which handles the $m = 1$ case.  If $m>1$ then
  $\sM:=\sL^{m-1}\otimes\pi_2^*\sO_{\bP^n}(1)$ is ample on $X$ and by
  (\ref{eq:4}) and Kodaira vanishing we have that
  $$
  H^i(X,\sL^m(-Z)) = H^i(X, \sL(-Z) \otimes \sL^{m-1}) \simeq 
  H^i(X,\omega_X \otimes \sM)=0
  $$
  and hence it follows from \autoref{prp:cone DB criterion} that $(CX,CZ)$ is indeed
  a \DB pair.
\end{example}

We will find the following lemma useful, \cf
\cite{EsnaultHodgeTypeOfSubvarieties,SchwedeEasyCharacterization}.

\begin{lemma}
  \label{lem.EmbeddedDBCriterionForPairs}
  Assume that $(X, Z)$ is a pair with $Z \subseteq X$ reduced schemes.  Assume
  further that $X \subseteq Y$ where $Y$ is smooth.  Let $\pi : \tld Y \to Y$ be a
  log resolution of both $X$ and $Z$ in $Y$ and set $\overline{X}$ and $\overline{Z}$
  to be the reduced preimages of $X$ and $Z$ in $\tld Y$ respectively.  Then
  $\DuBois{X, Z} \simeq \myR \pi_* \sI_{\overline{Z} \subseteq \overline{X}}$ where
  $\sI_{\overline{Z} \subseteq \overline{X}}$ is the ideal of $\overline{Z}$ in
  $\overline{X}$.
\end{lemma}
\begin{proof}
  Consider the diagram
  \[
  \xymatrix{%
    \DuBois{X,Z} \ar[d]_\alpha \ar[r] & \DuBois{X} \ar[d]_\beta \ar[r] & \DuBois{Z}
    \ar[d]_\gamma    \ar[r]^-{+1} & \\
    \myR \pi_* \DuBois{\overline{X}, \overline{Z}} \ar[r] & \myR \pi_*
    \DuBois{\overline{X}} \ar[r] & \myR \pi_* \DuBois{\overline{Z}} \ar[r]^-{+1} & \\
    \myR \pi_* \sI_{\overline{Z} \subseteq \overline{X}} \ar[r] \ar@{=}[u] & \myR
    \pi_* \sO_{\overline{X}} \ar[r] \ar@{=}[u] & \myR \pi_* \sO_{\overline{Z}}
    \ar[r]^-{+1} \ar@{=}[u] & }
  \]
  The vertical arrows $\beta$ and $\gamma$ are quasi-isomorphisms by
  \cite[Theorem 6.4]{KovacsSchwedeDuBoisSurvey} (also see \cite[Theorem 4.3]{SchwedeEasyCharacterization}) since $\overline{X}$ and $\overline{Z}$ are
  SNC and hence \DB.  The second row of equalities also follows since $\overline{X}$ and $\overline{Z}$ are \DB.  Then $\alpha$ is a quasi-isomorphism as well and hence the lemma
  follows.
\end{proof}

There are some situations when a pair being \DB implies that the ambient space is
also \DB. It is proved in \cite{GK13} that this happens if $X$ is Gorenstein, but
that $X$ being $\bQ$-Gorenstein is not sufficient. Another simple situation in which
this holds is the following:

\begin{lemma}
  \label{lem.DBPairCartierImpliesAmbientDB}
  Let $X$ be a reduced $\bC$-scheme essentially of finite type and $H$ a Cartier
  divisor.  If $(X, H)$ is a \DB pair then $X$ (and hence $H$) is also \DB.
\end{lemma}
\begin{proof}
  The statement is local and so we may assume that $X = \Spec R$ is affine.  We know
  that $\O_X(-H) \to \DuBois{X, H}$ is a quasi-isomorphism and thus so is $\O_X \to
  \DuBois{X, H} \tensor \O_X(H)$.  We will show that this map factors through
  $\O_X \to \DuBois{X}$ which will complete the proof by \cite[Theorem 2.3]{KovacsDuBoisLC1}.

  Embed $X \subseteq Y$ as a smooth scheme and let $\pi : \tld Y \to Y$ be a
  simultaneous log resolution of $(Y,X)$ and $(Y, H)$ with $\overline{X},
  \overline{H}$ the reduced total transforms of $X$ and $H$ respectively.  Then
  $\DuBois{X, H} = \myR \pi_* \sI_{\overline{H} \subseteq \overline{X}}$ by
  \autoref{lem.EmbeddedDBCriterionForPairs}.  Fix $X'$ to be the components of
  $\overline{X}$ which are not also components of $\overline{H}$ and we see that
  $\sI_{\overline{H} \subseteq \overline{X}} \simeq \O_{X'}(-\overline{H}|_{X'})$.
  Thus
  \[
  \DuBois{X, H} \tensor \O_X(H) \qis \myR \pi_* \O_{X'}((\pi^* H -
  \overline{H})|_{X'}).
  \]
  Since $\pi^* H - \overline{H}$ is effective, we obtain a map
  \[
  \DuBois{X} \qis \myR \pi_* \O_{\overline{X}} \to \myR \pi_* \O_{X'} \to \myR \pi_*
  \O_{X'}((\pi^* H - \overline{H})|_{X'}) \qis \DuBois{X, H} \tensor \O_X(H).
  \]
  This map obviously factors the quasi-isomorphism $\O_X \to \DuBois{X, H} \tensor
  \O_X(H)$ and hence the proof is complete.
\end{proof}

We recall properties of $\DuBois{X,Z}$ that we will need later.

\begin{lemma}
  \label{lem.BasicPropertiesOfDuBois}
  Let $X$ be a scheme over $\bC$ and $Z \subseteq X$ is a closed subscheme and $j : U
  = X \setminus Z \hookrightarrow X$ the complement of $Z$.  Then:
  \begin{enumerate}
  \item\label{item:1} If in addition $X$ is proper, then $H^i(X, \sI_{Z}) \to
    \bH^i(X, \DuBois{X,Z})$ is surjective for all $i \in \bZ$, \cite[Corollary
    4.2]{KovacsDBPairsAndVanishing}, \cite[Theorem
    6.22]{KollarKovacsSingularitiesBook}.
  \item\label{item:2} If $H$ is a general member of a base point free linear system
    then $\DuBois{X,Z} \tensor \O_H \qis \DuBois{H, H \cap Z}$, \cite[Proposition
    3.18]{KovacsDBPairsAndVanishing}, \cite[Theorem
    6.5(6)]{KollarKovacsSingularitiesBook}.
  \item\label{item:3} If $X = U \cup V$ is a decomposition into closed subschemes and
    $Z \subseteq X$ is another closed subscheme, then we have a distinguished
    triangle:
    \[
    \DuBois{U \cup V, Z} \to \DuBois{U, Z\cap U} \oplus \DuBois{V, Z\cap V} \to
    \DuBois{U \cap V, Z\cap {U \cap V}} \xrightarrow{+1}
    \]
    cf.\ \cite[Theorem 6.5(11)]{KollarKovacsSingularitiesBook}.
  \item\label{item:4} Let $X = U \cup V$ be a decomposition of $X$ into closed
    subschemes.
    Then $$\DuBois{U\cup V, V}\simeq \DuBois{U,U\cap V}.$$ cf.\
    \cite[Proposition 3.19]{KovacsDBPairsAndVanishing}, \cite[Theorem
    6.17]{KollarKovacsSingularitiesBook}.
  \end{enumerate}
\end{lemma}
\begin{proof}
  Parts (a) and (b) follow from the references in their statement.  For (c), the
  included reference only states the triangle in the case that $Z = \emptyset$.
  However, this more general version follows easily from the following diagram:
  \[
  \xymatrix{ \DuBois{U \cup V, Z} \ar[d] \ar@{.>}[r] & \DuBois{U, Z\cap U} \oplus
    \DuBois{V, Z\cap V} \ar[d] \ar@{.>}[r] & \DuBois{U \cap V, Z\cap {U \cap V}} \ar[d]
    \ar@{.>}[r]^-{+1} & \\
    \DuBois{U \cup V} \ar[d] \ar@{->}[r] & \DuBois{U} \oplus \DuBois{V} \ar[d]
    \ar@{->}[r] & \DuBois{U \cap V} \ar[d] \ar@{->}[r]^-{+1} &
    \\
    \DuBois{Z} \ar[d]_{+1} \ar@{->}[r] & \DuBois{Z\cap U} \oplus \DuBois{Z\cap V}
    \ar[d]_{+1}\ar@{->}[r] & \DuBois{Z\cap {U \cap V}} \ar[d]_{+1}
    \ar@{->}[r]^-{+1} & \\
    & & & }
  \]
  and the 9-lemma in triangulated categories \cite[B.1]{Kovacs11b}.

  For (d) consider the distinguished triangle $\xymatrix{\DuBois{U \cup V} \ar[r] &
    \DuBois{U} \oplus \DuBois{V} \ar[r] & \DuBois{U \cap V} \ar[r]^-{+1} & }$ of part
  (c) with $Z=\emptyset$ and \cite[Lemma 2.1]{KollarKovacsLCImpliesDB} implies that the
  left vertical arrow of the following diagram is an isomorphism:
    \[
    \xymatrix{ \DuBois{U \cup V, V} \ar[r]\ar[d]^{\simeq} & \DuBois{U \cup
        V}
      \ar[d]\ar[r] & \DuBois{V} \ar[r]^-{+1}\ar[d] &\\
      \DuBois{U, U \cap V} \ar[r] & \DuBois{U} \ar[r]& \DuBois{U
        \cap V} \ar[r]^-{+1} & \\
    }
    \]
    For more details see the proof in the references and replace
    $\underline\Omega^\times$ with $\DuBois{}$.
\end{proof}



We have one more lemma which constructs a natural exact triangle for Du~Bois pairs.

\begin{lemma}
\label{lem.TwoOfThreeDuBoisForPairs}
Let $X$ be a scheme and $W, Z \subseteq X$ subschemes.  Then there is a
distinguished triangle
\[
\DuBois{X, W \cup Z} \to \DuBois{X, Z} \to \DuBois{W, Z \cap W} \xrightarrow{+1}.
\]
In particular, because there is also a short exact sequence:
\[
0 \to \sI_{W \cup Z \subseteq X} \to \sI_{Z \subseteq X} \to \sI_{Z \cap W \subseteq
  W} \to 0,
\]
if any two of $\big\{ (X, W \cup Z), (X, Z), (W, Z \cap W) \big\}$ are Du Bois, so is
the third.
\end{lemma}
\begin{proof}
  We begin with a diagram of distinguished triangles as columns and rows (cf.\
  \cite[Theorem 6.5.11]{KollarKovacsSingularitiesBook} and \cite[Thm.~B1]{Kovacs11b}):
\[
\xymatrix{
  \DuBois{X, W \cup Z} \ar@{.>}[r] \ar[d]  & \DuBois{X, Z} \oplus \DuBois{X, W}
  \ar[d] \ar@{.>}[r]^-{-} & \DuBois{X, Z \cap W} \ar[d] \ar[r]^{+1} &\\
  \DuBois{X} \ar[d] \ar[r]^-{\id \oplus \id} &\DuBois{X} \oplus \DuBois{X} \ar[d]
  \ar[r]^-{-} & \DuBois{X} \ar[d] \ar[r]^{+1} &\\
  \DuBois{W \cup Z} \ar[r] & \DuBois{Z} \oplus \DuBois{W} \ar[r]_-{-} & \DuBois{Z
    \cap W} \ar[r]^{+1} &\\
}
\]
Note the second column of horizontal maps in the diagram above are each obtained by subtracting the canonical maps on each factor of the direct sum, hence the $-$ signs.
The octahedral axiom implies that there exists a diagram of distinguished
triangles,
\[
\xymatrix{
\DuBois{X, W \cup Z} \ar[d] \ar[r] & \DuBois{X, Z} \ar[d] \ar[r] & K^{\mydot}
\ar[d]^{\sim} \ar[r]^{+1} & \\
\DuBois{X, W} \ar[r] & \DuBois{X, Z \cap W} \ar[r] & K^{\mydot} \ar[r]_{+1} &
}
\]
We need to identify $K^{\mydot}$.  Notice that the bottom row also fits into another
diagram of distinguished triangles (cf.\ \cite[Thm.~B1]{Kovacs11b}):
\[
\xymatrix{
\DuBois{X, W} \ar[r] \ar[d] & \DuBois{X, Z \cap W} \ar[r] \ar[d] & K^{\mydot} \ar[r]^-{+1} \ar[d] & \\
\DuBois{X} \ar[r]^{\sim} \ar[d]  & \DuBois{X} \ar[r] \ar[d] & 0 \ar[r]^-{+1} \ar[d] & \\
\DuBois{W} \ar[r]  \ar[d]_{+1} & \DuBois{Z \cap W} \ar[r] \ar[d]\ar[d]_{+1} & \DuBois{W, Z \cap W}[1] \ar[d]_{+1} \ar[r]_-{+1} & \\
& & &
}
\]
Hence $K^{\mydot} \simeq \DuBois{W, Z \cap W}$ and the lemma follows.
\end{proof}

Finally, note that being \DB is a direct generalization of being rational for pairs.

\begin{theorem}\textnormal{(\cite[Corollary 5.6]{KovacsDBPairsAndVanishing},
    cf.\  \cite[Corollary 6.25]{KollarKovacsSingularitiesBook})}
  If $(X, D)$ is a rational pair then $(X, D)$ is also a \DB pair.
\end{theorem}

\section{An injectivity theorem} 

A key ingredient of the proof that Du Bois singularities are deformation invariant
was an injectivity theorem \cite[Theorem~3.3]{KovacsSchwedeDBDeforms}.  In this
section, we generalize that result to the context of pairs.

  \begin{lemma} \textnormal{(\cf \cite[Lemma 3.1]{KovacsSchwedeDBDeforms})}
    \label{lem.CompatibilityOfCyclicCoverAndDuBoisPairs}
    Let $X$ be a reduced scheme, $Z \subseteq X$ a reduced subscheme and $\sL$ a semi-ample line
    bundle.  Let $s \in \sL^n$ be a general global section for some $n\gg 0$ and take
    the $n^{\text{th}}$-root of this section as in \cite[Definition 2.50]{KollarMori}:
    $$
    \eta : Y = \sheafspec \bigoplus_{i = 0}^{n-1} \sL^{-i} \to X.
    $$
    Set $W = \eta^{-1}(Z)$ (with the induced scheme structure).  Note that the
    restriction satisfies \mbox{$\eta|_W : W = \sheafspec \bigoplus_{i = 0}^{n-1}
      \sL^{-i}|_Z \to Z$}.  Then as before, writing $\eta_*=\myR\eta_*$,
    \[
    \eta_* \DuBois{Y,W} \simeq \DuBois{X,Z}\otimes \eta_*\sO_Y \simeq \bigoplus_{i =
      0}^{n-1} (\DuBois{X,Z} \tensor \sL^{-i}),
    \]
    and this direct sum decomposition is compatible with the decomposition $\eta_*
    \O_Y = \bigoplus_{i = 0}^{n-1} \sL^{-i}$.
  \end{lemma}
  \begin{proof}
    Although not explicitly stated, it is easy to see that \cite[Lemma
    3.1]{KovacsSchwedeDBDeforms} is functorial in that it is compatible with the map
    $Z \to X$.  Then by applying
    \autoref{lem.BasicPropertiesOfDuBois}\autoref{item:2}, the result follows from
    the diagram:
    \[
    \xymatrix{ \eta_* \DuBois{Y,W} \ar[r] & \eta_* \DuBois{Y} \ar[r] & \eta_*
      \DuBois{W}
      \ar[r]^{+1} &\\
      \DuBois{X,Z} \tensor \eta_* \O_Y \ar@{.>}[u] \ar[r] & \DuBois{X}\ar[u]_{\simeq}
      \tensor \eta_* \O_Y \ar[r] & \DuBois{Z}\ar[u]_{\simeq} \tensor \eta_* \O_Y
      \simeq \DuBois{Z}\ar[u] \tensor_{\O_Z} \eta_* \O_Z \ar[r]^-{+1} &.  }
    \]
  \end{proof}

  Setting $\DBDual{X,Z} = \myR \sHom_{\O_X}^{\mydot}(\DuBois{X,Z}, \omega_X^{\mydot})$ as in
  \autoref{eq.DBDualDefinition}, we easily obtain the following.

  \begin{theorem}
    \label{thm.MainInjectivityForPairs}
    Let $X$ be a reduced scheme over $\bC$ and $Z \subseteq X$ a reduced subscheme.  Then the
    natural map
    \[
    \Phi^j : \myH^j(\DBDual{X,Z}) \into \myH^j(\myR \sHom_{\O_X}(\sI_Z,
    \omega_X^{\mydot}))
    \]
    is injective for every $j\in\bZ$.
  \end{theorem}
  \begin{proof}
    The proof is essentially the same as in \cite[Theorem
    3.3]{KovacsSchwedeDBDeforms} so we only sketch it briefly.  First, since the
    question is local and compatible with restricting to an open subset, we may
    assume that $X$ is projective with ample line bundle $\sL$.  It follows from
    taking a cyclic cover with respect to a general section of $\sL^n$, for $n \gg
    0$, and applying \autoref{lem.BasicPropertiesOfDuBois}\autoref{item:1} and
    \autoref{lem.CompatibilityOfCyclicCoverAndDuBoisPairs} that $H^j(X, \sI_Z \otimes
    \bigoplus_{i = 0}^{n-1} \sL^{-i}) \to \bH^j(X, \DuBois{X,Z} \tensor \bigoplus_{i
      = 0}^{n-1} \sL^{-i})$ surjects for all $j \geq 0$.  Therefore $H^j(X, \sI_Z
    \otimes \sL^{-i}) \to \bH^j(X, \DuBois{X,Z} \tensor \sL^{-i})$ surjects for all
    $i,j \geq 0$.

    By an application of Serre-Grothendieck duality we obtain an injection
    \begin{equation}
    \label{eq.InjectPreSS}
    \bH^{j}(X, \DBDual{X,Z} \tensor \sL^{i}) \into \bH^{j}(X, \myR
    \sHom_{\O_X}^{\mydot}(\sI_Z, \omega_X^{\mydot}) \tensor \sL^{i})
    \end{equation}
    for all $i,j \geq 0$.  But for $i \gg 0$, by Serre vanishing, we obtain that
    \begin{equation}
    \label{eq.InjectPostSS}
    H^0(X, \myH^j (\DBDual{X,Z}) \tensor \sL^{i}) \into H^{0}(X, \myH^j (\myR
    \sHom_{\O_X}^{\mydot}(\sI_Z, \omega_X^{\mydot}) )\tensor \sL^{i})
    \end{equation}
    is injective as well (since the spectral sequence computing
    \autoref{eq.InjectPreSS} degenerates).  On the other hand, if $\myH^j
    (\DBDual{X,Z}) \to \myH^j (\myR \sHom_{\O_X}^{\mydot}(\sI, \omega_X^{\mydot}))$
    is not injective, for some $i \gg 0$ neither is \autoref{eq.InjectPostSS}.  This
    completes the proof.
  \end{proof}

\section{Deformation of Du Bois pairs}

In \cite[Corollary~4.2]{KovacsSchwedeDBDeforms}, 
we showed the following result: Let $f : X \to B$ be a flat proper family over a
smooth curve $B$ with a fiber $X_0$, $0 \in B$, having Du Bois singularities.  Then
there is an open neighborhood $0 \in U \subseteq B$ such that the fibers $X_u$ have
Du Bois singularities for $u \in U$.  In this section, we generalize this result to
Du Bois pairs.  We mimic our previous approach as much as possible.

First we need a lemma which is presumably well known but for which we know no
reference.

\begin{lemma}
  \label{lem.transversality}
  Let $X$ be a reduced scheme and $Z\subseteq X$ a reduced subscheme with ideal
  sheaf $\sI_Z$. Further let $H\subseteq X$ be an effective Cartier divisor with ideal
  sheaf $\sI_H$ such that $H$ does not contain any irreducible components of either
  $X$ or $Z$. Then $$\sI_H\cap\sI_Z =\sI_H\cdot \sI_Z.$$
\end{lemma}

\begin{proof}
This is left as an exercise to the reader.  Earlier versions of this paper, which are available on the arXiv, also contain a detailed proof.
\end{proof}
Now we prove that if a special fiber supports a Du Bois pair, so does the total space
near that fiber.  Recall that effective Cartier divisors on a
possibly non-normal scheme are simply subschemes locally defined by a single
non-zero-divisor near every point.

\begin{theorem}
  \label{thm.MainTheoremDeformation}
  Let $X$ be a reduced scheme essentially of finite type over $\bC$, $Z \subseteq X$ a
  reduced subscheme and $H$ a reduced effective Cartier divisor on $X$ that does not
  contain any component of $Z$.  If $(H, Z\cap H)$ is a \DB pair, then $(X,Z)$ is a
  \DB pair near $H$.  It then follows from \autoref{lem.TwoOfThreeDuBoisForPairs} that $(X, Z \cup H)$ is Du Bois near $H$.
\end{theorem}
\begin{proof}
  We follow very closely the proof of \cite[Theorem 3.2]{KovacsDuBoisLC2} and \cite[Theorem 4.1]{KovacsSchwedeDBDeforms}, which are based on \cite{ElkikDeformationsOfRational}.
  Choose a closed point ${\bq}$ of $X$ contained within $H$.  It is sufficient
  to prove that $(X, Z)$ is \DB at ${\bq}$.  Let $R$ denote the stalk $\O_{X, {\bq}}$
  and replace $X$ by $\Spec R$.  Choose $f \in R$ to denote a defining equation of
  $H$ in $R$.  Consider the following diagram whose rows are distinguished triangles
  in $D^{b}_{\coherent}(X)$:

  \begin{equation}
    \label{eq.DiagramToDual}
    \begin{aligned}
      \xymatrix@C=4em {%
        I_Z \ar[d] \ar[r]^{\times f} & I_Z \ar[d] \ar[r] & I_Z/ (f\cdot
        I_Z)\ar[d]_{\rho} \ar@/^2em/[dd]|!{[d];[dr]}\hole^(.7){\qis} \ar[r]^-{+1}
        & \\
        \DuBois{X,Z} \ar[r]_{\times f} & \DuBois{X,Z} \ar[r] & A^{\mydot}
        \ar@{.>}[d]_{\tau}    \ar[r]^-{+1} & \\
        & & \DuBois{H, Z\cap H}
        &\\
      }
    \end{aligned}
  \end{equation}
  where $A^{\mydot}$ is the term completing the second row to a distinguished
  triangle.  We claim we have a map $\tau$ as above such that $\tau \circ \rho$ is a
  quasi-isomorphism.  Certainly we have a diagram with distinguished triangles for
  rows and columns
  \[
  \xymatrix@R=4pt{ \DuBois{X,Z} \ar[dd] \ar[r]^{\times f} & \DuBois{X,Z} \ar[dd]
    \ar[r] & A^{\mydot}
    \ar@/_0.5pc/@{.>}[drrr]^{\tau} \ar[dd] \ar[r]^{+1} &\\
    & & & & & \DuBois{H, Z\cap H} \ar[dd]\\
    \DuBois{X} \ar[dd] \ar[r]^{\times f} & \DuBois{X} \ar[dd] \ar[r] & B^{\mydot}
    \ar[dd] \ar[r]^{+1} \ar@/_0.5pc/@{.>}[drrr]^{\kappa} & & &   \\
    & & & & & \DuBois{H}\ar[dd] \\
    \DuBois{Z} \ar[r]^{\times f} \ar[dd]_{+1} & \DuBois{Z} \ar[r] \ar[dd]_{+1}&
    C^{\mydot} \ar[r]^{+1} \ar[dd]_{+1}\ar@/_0.5pc/@{.>}[drrr]_{\mu} & & &  \\
    & & & & & \DuBois{Z \cap H} \ar[dd]^-{+1}\\
    & & & & & \\
    & & & & & \\
  }
  \]
  and the existence of $\tau$ follows immediately from the existence of $\kappa$ and
  $\mu$ whose existence follows from the proof of \cite[Theorem
  4.1]{KovacsSchwedeDBDeforms}.  Note that the assumptions imply that $H|_Z = H \cap
  Z$ is a Cartier divisor on $Z$, so we may indeed use \cite[Theorem
  4.1]{KovacsSchwedeDBDeforms} for both $X$ and $Z$.
  Since $I_Z / (f \cdot I_Z) = I_Z / ( (f) \cap I_Z)$ by \autoref{lem.transversality}
  and because $(H, Z \cap H)$ is a \DB pair, we see $\tau \circ \rho$ is an
  isomorphism as claimed.

  Next we apply the Grothendieck duality functor $\myD(\blank) = \myR
  \Hom_R^{\mydot}(\blank, \omega_R^{\mydot})$ to \autoref{eq.DiagramToDual} and take
  cohomology:
  \[
  {\small \xymatrix@C=16pt{%
      \dots \ar@{<-}[r] & \myH^i(\myD(I_Z)) \ar@{<-_{)}}[d]^{\Phi^i} &
      \ar[l]_-{\times f} \myH^i(\myD(I_Z)) \ar@{<-_{)}}^{\Phi^i}[d]
      \ar@{<-}[r]^-{\delta_i} & \myH^i(\myD(I_Z/(f \cdot I_Z))
      \ar@{<<-}[d]^-{\gamma_i} \ar@{<-}[r]^-{\alpha_i} & \myH^{i-1}(\myD(I_Z))
      \ar@{<-_{)}}[d]^{\Phi^{i-1}} \ar@{<-}[r]^-{\times f} &
      \myH^{i-1}(\myD(I_Z)) \ar@{<-_{)}}[d]^{\Phi^{i-1}}  & \ar[l]\cdots  \\
      \dots \ar@{<-}[r] & \myH^i(\DBDual{X,Z}) \ar@{<-}[r]_{\times f} & \myH^i
      (\DBDual{X,Z}) \ar@{<-}[r] & \myH^i(\myD(A^{\mydot})) \ar@{<-}[r]_-{\beta_i} &
      \myH^{i-1}(\DBDual{X,Z} \ar@{<-}[r]_-{\times f}) & \myH^{i-1}(\DBDual{X,Z})}}
  \]
  where the $\Phi^\mydot$ are injective by \autoref{thm.MainInjectivityForPairs} and
  $\gamma_i$, which was obtained from $\rho$, is surjective since $\tau \circ \rho$
  is an isomorphism.

  The proof now follows exactly as the main theorem of \cite{KovacsSchwedeDBDeforms},
  or dually of \cite[Theorem 3.2]{KovacsDuBoisLC2}.  Fix $z \in \myH^{i-1}
  (\myD(I_Z))$.  Pick $w \in \myH^i(\myD(A^{\mydot}))$ such that $\alpha_i(z) =
  \gamma_i(w)$.  Since $\delta_i(\alpha_i(z)) = 0$ and $\Phi^i$ is injective, it
  follows that there exists a $u \in \myH^{i-1} (\DBDual{X,Z})$ such that $\beta_i(u)
  = w$.  Therefore, $\alpha_i(\Phi^{i-1}(u)) = \alpha_i(z)$ and so
  \begin{equation}
    \label{eq.diffInFTimes}
    z - \Phi^{i-1}(u) \in f \cdot \myH^{i-1} (\myD(I_Z)).
  \end{equation}

  Now, fix $E_{i-1}$ to be the cokernel of $\Phi^{i-1}$ and set $\overline{z} \in
  E_{i-1}$ to be the image of $z$.  Equation (\ref{eq.diffInFTimes}) then guarantees
  that $\overline{z} \in f \cdot E_{i-1}$.  The multiplication map $\xymatrix{E_{i-1}
    \ar[r]^{\times f} & E_{i-1}}$ is then surjective and so Nakayama's lemma
  guarantees that $\Phi^{i-1}$ is also surjective.  Therefore $\DBDual{X,Z} \to
  \myD(I_Z)$ is a quasi-isomorphism which implies that $(X, Z)$ is a \DB pair.
\end{proof}

We immediately obtain:

\begin{corollary}\label{cor.CurveDBOpenness}
  Let $f : X \to B$ be a flat proper family of varieties over a smooth
  one-dimensional scheme $B$ essentially of finite type over $\bC$ (for instance, a
  smooth curve).  Further let $Z \subseteq X$ be a subscheme such that no component
  of $Z$ is contained in any component of any fiber of $f$ and $b\in B$ a closed
  point such that $(X_b,Z_b)$ is a \DB pair.  Then there exists a neighborhood $b\in
  U\subseteq B$ such that
  \begin{enumerate}
  \item $(X, Z)$ is \DB over $U$, and
  \item the fibers $(X_u, Z_u)$ are \DB for all $u \in U$.
  \end{enumerate}
\end{corollary}
\begin{proof}
  The non-\DB locus $T$ of $(X, Z)$ is closed, and since $f$ is proper, $f(T)$ is
  also closed.  Hence (a) follows from \autoref{thm.MainTheoremDeformation} and by
  replacing $B$ with an open set, we may assume that $(X, Z)$ is \DB. Then the
  Bertini type theorem \autoref{lem.BasicPropertiesOfDuBois}\autoref{item:2} implies
  that (b) follows after possibly shrinking $U$.
\end{proof}


\begin{corollary}
  \label{cor.DBDeformsArbitraryBase}
  Let $f : X \to B$ be a flat proper family of varieties over a smooth scheme $B$
  essentially of finite type over $\bC$.  Further let $Z \subseteq X$ be a subscheme
  which is also flat over $B$ and $b\in B$ a closed point such that $(X_b,Z_b)$ is a
  \DB pair.  Then there exists a neighborhood $b\in U\subseteq B$ such that $(X, Z)$
  is \DB over $U$.
\end{corollary}
\begin{proof}
  We may assume that $B$ is affine and let $d = \dim B$.  We first show that $(X, Z)$
  itself is \DB in a neighborhood of $(X_b, Z_b)$.
  Let $H_1, \ldots, H_d$ be general smooth subschemes going through $b$ whose local
  defining equations generate the maximal ideal of $b$ (i.e., locally analytically
  they are coordinate hyperplanes).  The pair $(X_b, Z_b) = (X_{H_1 \cap H_2 \cap
    \ldots \cap H_d}, Z_{H_1 \cap H_2, \cap \ldots \cap H_d})$ is \DB by assumption,
  hence since $X_b = X_{H_1 \cap \ldots \cap H_d}$ is a hypersurface in $X_{H_2 \cap
    \ldots H_d}$ it follows that the pair $(X_{H_2 \cap \ldots \cap H_d}, Z_{H_2,
    \cap \ldots \cap H_d})$ is \DB in a neighborhood of $X_b$ by
  \autoref{cor.CurveDBOpenness}.  Let $W_1$ denote the non-\DB locus of $(X_{H_2 \cap
    \ldots \cap H_d}, Z_{H_2, \cap \ldots \cap H_d})$. Since $W_1$ is closed and $f$
  is proper, we see that $f(W_1)$ is closed in $H_2 \cap \ldots \cap H_d$ and doesn't
  contain $b$.  Shrinking $B$ if necessary, we may assume that $W_1$ is empty.  Next
  observe that $H_2 \cap \ldots \cap H_d$ is a hypersurface in $H_3 \cap \ldots \cap
  H_d$ and so again we see that $(X_{H_3 \cap \ldots \cap H_d}, Z_{H_3, \cap \ldots
    \cap H_d})$ is \DB in a neighborhood of $X_{H_2 \cap \ldots \cap H_d}$ by
  \autoref{cor.CurveDBOpenness}.  Set $W_2$ to be the non-\DB locus of $(X_{H_3 \cap
    \ldots \cap H_d}, Z_{H_3, \cap \ldots \cap H_d})$ and note that $f(W_2)$ does not
  intersect $H_2 \cap \ldots \cap H_d$.  We shrink $B$ again if necessary so that
  $W_2 = \emptyset$.  Iterating this procedure proves the statement.
\end{proof}

In order to extend \autoref{cor.CurveDBOpenness}(b) to families over arbitrary
dimensional bases we need the following lemma.

\begin{lemma}
  \label{lem.DBdiscriminant} Let $f : X \to B$ be a flat proper family of varieties
  over a scheme $B$ essentially of finite type over $\bC$.  Further let $Z \subseteq
  X$ be a subscheme which is also flat over $B$ and assume that $(X, Z)$ is a \DB
  pair.  Then
  $$
  V=\{ b\in B \vert (X_b,Z_b) \text{ is a \DB pair} \}
  $$
  is a constructible set in $B$. Furthermore, if $B$ is smooth, then $V$ is open.
\end{lemma}

\begin{proof} We use induction on the dimension of $B$.

  Let $\pi : B' \to B$ be a resolution of singularities and consider the base change
  $f':X'=X_{B'}\to B'$, a flat proper family over $B'$ and with $Z'=Z_{B'}\subseteq
  X'$ subscheme flat over $B'$.  Notice that all the fibers of $f':X' \to B'$ appear
  as fibers of $f:X \to B$ (up to harmless field extension), so $b'\in
  V'=\pi^{-1}(V)\subseteq B'$ if and only if the fiber $(X'_{b'}, Z'_{b'})$ is a \DB
  pair. It follows from \autoref{cor.DBDeformsArbitraryBase} that by replacing $B'$
  with an open subset we may assume that $(X',Z')$ is a \DB pair. It also follows
  that it is enough to prove the statement over a smooth irreducible base. Indeed, that implies
  that $V'$ is open in $B'$ and hence $V=f(V')$ is constructible.

  To simplify notation we will replace $B$ with $B'$ and assume that $B$ is smooth
  and irreducible, but use the inductive hypothesis without these additional
  assumptions.

  The Bertini type statement \autoref{lem.BasicPropertiesOfDuBois}\autoref{item:2} implies that,
  if $V \neq \emptyset$, there is a dense open subset $U\subseteq B$ contained in
  $V$.  The case $\dim B=1$ follows immediately via the fact that in a curve any set
  containing a dense open set is itself open.

  In general, it follows that $\dim (B\setminus U) <\dim B$ so by induction
  $V\setminus U$ is a constructible set in $B\setminus U$ and hence $V$ is
  constructible in $B$. In the case of a smooth base
  \autoref{cor.DBDeformsArbitraryBase} implies that $V$ is stable under
  generalization and since we have just proved that it is constructible it follows
  that it is open.
\end{proof}

\begin{corollary}
  \label{cor.DBDeformsArbitraryBaseFibers} Let $f : X \to B$ be a flat proper family
  of varieties over a smooth scheme $B$ essentially of finite type over $\bC$.
  Further let $Z \subseteq X$ be a subscheme which is also flat over $B$ and $b\in B$
  a closed point such that $(X_b,Z_b)$ is a \DB pair.  Then there exists a
  neighborhood $b\in U\subseteq B$ such that $(X_u, Z_u)$ is a \DB pair for all $u\in
  U$.
\end{corollary}

\begin{proof}
  Observe that the non-\DB locus $W$ of $(X, Z)$ is closed in $X$ and since $f$ is
  proper, $f(W)$ is also closed in $B$.  Note that $f(W)$ does not contain $b$ so it
  also does not contain the generic point of $B$.  Hence by replacing $B$ by a
  neighborhood $U \subseteq B$ of $b\in B$, we may assume that $(X, Z)$ is \DB. Then
  the statement follows from \autoref{lem.DBdiscriminant}.
\end{proof}

\begin{remark}
  One can recover special cases of inversion of adjunction for log canonicity
  \cite{KawakitaInversion} easily from \autoref{thm.MainTheoremDeformation}. For
  instance, let $(X, D+H)$ be a pair with $K_X, D$ and $H$ Cartier and assume
  that $(H, D|_H)$ is slc or equivalently \DB \cite{KollarKovacsSingularitiesBook}.
  Then $(X, D+H)$ is \DB or equivalently lc by \autoref{thm.MainTheoremDeformation}.
\end{remark}

\section{A result of Koll\'ar-Kov\'acs for pairs}
\label{sec.KollarKovacsCMForPairs}

The following was shown in \cite[Theorem 7.12]{KollarKovacsLCImpliesDB}: Let $f : X
\to B$ be a flat projective family of varieties with Du Bois singularities.  Then if
$B$ is connected and the general fiber is Cohen-Macaulay, then all the fibers are
Cohen-Macaulay.

We would like to generalize this to the context of Du Bois pairs, at least in the
case when $Z$ is a divisor.  We recommend for the reader to have a copy of
\cite{KollarKovacsLCImpliesDB} available when reading this section as we refer to a
number of lemmas therein.  We begin by generalizing a result of Du Bois and Jarraud
to pairs cf.\ \cite{DuBoisJarraudAPropertyOfCommutation}, \cite[Th\'eorm\`e 4.6]{DuBoisMain}.

\begin{theorem}\label{thm.DuBoisJarraudForPairs}
  Let $f : X \to B$ be a flat proper morphism between schemes of finite type over
  $\bC$.  Assume that $B$ is smooth and let $Z \subseteq X$ be a subscheme that is
  flat over $B$.  Further assume that the geometric fibers $(X_b, Z_b) \to b$ are Du
  Bois.  Then $\myR^i f_* \sI_Z$ is locally free of finite rank and furthermore
  compatible with base change for all $i$, in other words $(\myR^i f_* \sI_Z)_T
  \simeq \myR^i f_* \sI_{Z_T}$ for any morphism $T\to B$.
\end{theorem}
\begin{proof}
  For some $b \in B$, let $\frm$ be the maximal ideal of $\sO_{B,b}$ and $S=S_n=\Spec
  \sO_{B,b}/\frm^{n+1}$ for $n\in\bN$.  Further let $\sI_{Z_b}$ resp.\ $\sI_{Z_S}$
  denote the ideal sheaf of $Z_b$ in $X_b$ resp.\ $Z_S$ in $X_S$.  Consider the
  following commutative diagram:
  $$
  \xymatrix{%
    H^i_c((X_S \setminus Z_S)^{\an}, \mathbb{C}) \ar[rd]^\alpha\ar[dd]_\lambda & &
    \ar[ld]_\beta H^i(X_S, \sI_{Z_S}) \ar[dd]^\nu \\ & H^i(X_S^{\an}, \sI_{Z_S})
    \ar[dd]_\mu \\ H^i_c((X_b \setminus Z_b)^{\an}, \mathbb{C}) \ar[rd]_\gamma & &
    \ar[ld]^\delta H^i(X_b, \sI_{Z_b}) \\ & H^i(X_b^{\an}, \sI_{Z_b}). }
  $$
  Observe that $\lambda$ is an isomorphism since $X_S$ and $X_b$ have the same
  support.  By \cite[Theorem 4.1]{KovacsDBPairsAndVanishing} cf.\ \cite[Theorem
  6.8]{KollarKovacsSingularitiesBook} $\gamma$ is surjective, so
  $\gamma\circ\lambda=\mu\circ\alpha$ is surjective and hence $\mu$ is surjective.
  By Serre's GAGA principle \cite{MR0082175} $\beta$ and $\delta$ are isomorphisms
  and hence $\nu$ is surjective. Finally, the statement follows by Cohomology and
  Base Change \cite[7.7]{EGAIII2}.
\end{proof}

Next we prove the analogue of the main flatness and base change result of
Koll\'ar-Kov\'acs for \DB pairs \cite[Theorem 7.9]{KollarKovacsLCImpliesDB} .

\begin{theorem}
  \label{thm.MainBaseChangeKollarKovacs} Let $f : X \to B$ be a flat projective
  morphism between schemes of finite type over $\bC$ and assume that $B$ is smooth.
  Let $Z \subseteq X$ be a closed subscheme that is flat over $B$ and $\sL$ a relatively
  ample line bundle on $X$.  Assume $(X, Z)$ is Du~Bois. Then
  \begin{enumerate}
  \item the sheaves $\myH^{-i}(\myR\sHom_{\O_X}^{\mydot}(\sI_Z,
    \omega_{f}^{\mydot}))$ are flat over $B$ for all $i$.
  \item the sheaves $f_* \big(\myH^{-i}\big(\myR\sHom_{\O_X}^{\mydot}(\sI_Z,
    \omega_{f}^{\mydot})) \tensor \sL^q\big)$ are locally free and compatible with
    arbitrary base change for all $i > 0$ and $q \gg 0$.
  \item for any base change $\vartheta : T \to B$ and for all $i > 0$,
    \[ \big(\myH^{-i}(\myR\sHom_{\O_X}^{\mydot}(\sI_Z, \omega_{f}^{\mydot}))\big)_T
    \qis \myH^{-i}(\myR\sHom_{\O_{X_T}}^{\mydot}(\sI_{Z_T}, \omega_{f_T}^{\mydot})).
    \]
  \end{enumerate}
\end{theorem}
\begin{proof}
  We follow the proof of \cite[Theorem 7.9]{KollarKovacsLCImpliesDB}.  We may assume
  that $B = \Spec R$ is affine and hence that $\sL^m$ is globally generated for $m
  \gg 0$.  For such an $m \gg 0$, choose a general section $\sigma \in H^0(X, \sL^m)$
  and consider the cyclic cover induced by $\sigma$:
  \[
  \sA = \bigoplus_{j = 0}^{m-1} \sL^{-j} \simeq \bigoplus_{j = 0}^{m-1} \sL^{-j} t^j
  \Big/(t^m - \sigma).
  \]
  Set $h : Y = \Spec_X \sA \to X$, and $Z_Y = h^{-1} Z$ with the induced reduced
  scheme structure.  Then the geometric fibers of the composition $(Y, Z_Y)\to B$ are
  also \DB by \cite[Corollary 6.21]{KollarKovacsSingularitiesBook}.  Note that by
  construction $\sI_{Z_Y} = \bigoplus_{j = 0}^{m-1} \sI_Z \tensor \sL^{-j}$.  Hence
  $\myR^i h_* \sI_{Z_Y}$ is locally free of finite rank and compatible with arbitrary
  base change by \autoref{thm.DuBoisJarraudForPairs}.  It follows that the summands
  of these modules, the $\myR^i f_* (\sI_Z \tensor \sL^{-j})$, are also locally free
  and compatible with base change. Since we may choose $m$ arbitrarily large, this
  holds for all $j\in \bN$.  It follows immediately that $\sHom_{\O_B}(\myR^i f_*
  (\sI_Z \tensor \sL^{-j}), \O_B)$ is also locally free and compatible with base
  change.

  By Grothendieck duality and \cite[Lemma~7.3]{KollarKovacsLCImpliesDB} (cf. proof of
  \cite[Lemma~7.2]{KollarKovacsLCImpliesDB}) it follows that
  \begin{multline*}
    \sHom_{\O_B}(\myR^i f_* (\sI_Z \tensor \sL^{-q}), \O_B ) \simeq \\ \simeq f_*
    \myH^{-i} (\myR \sHom_{\O_X}^{\mydot}(\sI_Z, \omega_{f}^{\mydot} \tensor \sL^{q}))
    \simeq f_* \big(\myH^{-i} (\myR \sHom_{\O_X}^{\mydot}(\sI_Z, \omega_{f}^{\mydot}
    )\tensor \sL^{q})\big)
  \end{multline*}
  and hence (b) is proven.  Just as in \cite[Theorem~7.9]{KollarKovacsLCImpliesDB},
  (a) follows from (b) by an argument similar to \cite[Chapter III, Theorem
  9.9]{Hartshorne}.

  Finally we prove (c).  Since $f: X \to B$ is projective and $B$ is affine, we may
  factor $f$ as $X \xrightarrow{i} \bP^n_B \xrightarrow{\pi} B$.  It then suffices to
  show that
  \[
  \varrho^{-i} : \big(\myH^{-i}(\myR\sHom_{\O_{\bP^n_B}}^{\mydot}(\sI_Z,
  \omega_{\pi}[n]))\big)_T \to \myH^{-i}(\myR\sHom_{\O_{\bP^n_T}}^{\mydot}(\sI_{Z_T},
  \omega_{\pi}[n]))
  \]
  is an isomorphism.  As in \cite[Theorem 7.9]{KollarKovacsLCImpliesDB}, we proceed
  by descending induction on $i$ (the base case where $i \gg 0$ is obvious).  We
  observe that $\sI_Z$ is flat since so are $\sO_X$ and $\sO_Z$ and assume that
  $\varrho^{-(i+1)}$ is an isomorphism by induction.  Since
  $\myH^{-i}\big(\myR\sHom_{\O_{\bP^n_B}}^{\mydot}(\sI_Z, \omega_{\pi}[n])\big)$ is
  flat by (a), we may apply \cite[Theorem 1.9]{AltmanKleimanCompactifying} which
  completes the proof.
\end{proof}

The following is the analog of \cite[Theorem 7.11]{KollarKovacsLCImpliesDB} for
pairs.

\begin{theorem}
  \label{thm.SkFiberVsSkSheaf}
  Let $f : X \to B$ be a flat projective morphism between schemes of finite type over
  $\bC$.  Assume that $B$ is smooth and let $Z \subseteq X$ be a subscheme that is
  flat over $B$.  Let $x \in X$ be a closed point and let $b = f(x)$.  Then
  $\sI_{Z_b} \subseteq \O_{X_b}$ is \sk{k} at $x$ if and only if
  \begin{equation}
    \label{eq:3}
    \big(\myH^{-i} (\myR\sHom_{\O_X}^{\mydot}(\sI_Z, \omega_{f}^{\mydot}))\big)_y = 0
  \end{equation}
  for $i<\min(k+\dim \overline{\{y\}}, \dim_x X)$ for all $y\in X_b$ such that
  $x\in\overline{\{y\}}$. In particular, $\sI_{Z_b}$ is \sk{k} if and only if
  \eqref{eq:3} holds for $i<\min(k+\dim \overline{\{y\}},\dim_x X)$ for all $y\in
  X_b$ (not restricted to closed points).
\end{theorem}

First we prove a lemma.

\begin{lemma}
  Let $X$ be a scheme that admits a dualizing complex $\omega_X^\mydot$. Let $x\in X$
  and $\sF$ a coherent sheaf on $X$. Then $\sF$ is \sk{k} at $x\in X$ if and only if
  $$
  \big(\myH^{-i} (\myR\sHom_{\O_X}^{\mydot}(\sF, \omega_{X}^{\mydot}))\big)_y = 0
  $$
  for $i<\min(k, \dim \sF_{y})+\dim \overline{\{y\}}$ for all $y\in X$ such that $x\in\overline{\{y\}}$.
\end{lemma}

\begin{proof}
  This is a direct consequence of local duality \cite{HartshorneResidues} and the cohomological criterion for depth, see for instance \cite[Proposition~3.2]{MR2918171}.
\end{proof}

\begin{proof}[Proof of \autoref{thm.SkFiberVsSkSheaf}]
  By the lemma,  $\sI_{Z_b} \subseteq \O_{X_b}$
  is \sk{k} at $x$ if and only if
  \[
  \big(\myH^{-i} (\myR\sHom_{\O_{X_b}}^{\mydot}(\sI_{Z_b},
  \omega_{X_b}^{\mydot}))\big)_y = 0
  \]
  for $i<\min (k, \dim (\sI_{Z_b})_y)+\dim \overline{\{y\}}=\min (k + \dim
  \overline{\{y\}}, \dim_x X)$ for all $y\in X_b$ such that $x\in\overline{\{y\}}$.
  By \autoref{thm.MainBaseChangeKollarKovacs},
  \begin{multline*}
    \big(\myH^{-i} (\myR \sHom_{\O_{X_b}}^{\mydot}(\sI_{Z_b},
    \omega_{X_b}^{\mydot}))\big)_y \simeq
    \Big(\big(\myH^{-i}(\myR\sHom_{\O_{X}}^{\mydot}(\sI_{Z},
    \omega_{f}^{\mydot}))\big)_b\Big)_y \simeq \\ \simeq
    \Big(\big(\myH^{-i}(\myR\sHom_{\O_{X}}^{\mydot}(\sI_{Z},
    \omega_{f}^{\mydot}))\big)_y\Big)_b.
  \end{multline*}
  But notice that the right side is zero if and only if
  $\myH^{-i}(\myR\sHom_{\O_{X}}^{\mydot}(\sI_{Z}, \omega_{f}^{\mydot}))\big)_y$ is
  zero by Nakayama's lemma.  This implies the desired statement.
\end{proof}

Finally, we describe how the \sk{k} condition behaves for pairs in families where the fibers are \DB.

\begin{theorem}
  \label{thm.SkForDBFamilies} Let $f: (X, Z) \to B$ be a flat projective family
  with $\O_Z$ (and hence $\sI_Z$) flat over $B$ as well.  Assume that all the fiber pairs $(X_b, Z_b)$ are \DB.  Assume also that $B$ is
  connected and the generic fibers $(\sI_Z)_{\gen}$ are \sk{k}, then all the fibers
  $(\sI_Z)_b$ are \sk{k}.
\end{theorem}
\begin{proof}
  By working with one component of $B$ at a time, we may assume that $B$ is
  irreducible and hence that $X$ is equidimensional.  If $(\sI_Z)_b \simeq \sI_{Z_b}$
  (by flatness of $\O_Z$) is not \sk{k} at some point $y \in X_b$, then by
  \autoref{thm.SkFiberVsSkSheaf}, $\myH^{-i} (\myR\sHom_{\O_X}^{\mydot}(\sI_Z,
  \omega_{f}^{\mydot})) \neq 0$ near $y$ for some $i < \min( k + \dim
  \overline{\{y\}}, \dim X)$.  Fix an irreducible component $W \subseteq \supp
  \big(\myH^{-i} (\myR\sHom_{\O_X}^{\mydot}(\sI_Z, \omega_{f}^{\mydot})) \big)$ and
  observe that $\dim W_b$ is constant for $b\in B$ since $\myH^{-i}
  (\myR\sHom_{\O_X}^{\mydot}(\sI_Z, \omega_{f}^{\mydot}))$ is flat by
  \autoref{thm.MainBaseChangeKollarKovacs}(a).  However, in that case it follows that
  $\myH^{-i} (\myR\sHom_{\O_X}^{\mydot}(\sI_Z, \omega_{f}^{\mydot}))$ is non-zero
  near some point $\eta \in X_{\gen}$ such that $\dim\overline{\{\eta\}} =
  \dim\overline{\{y\}}$ which contradicts the assumption that the generic fiber is
  \sk{k} by \autoref{thm.SkFiberVsSkSheaf}.
\end{proof}

We immediately obtain the following.

\begin{corollary}
  \label{cor.CMClosedDuBoisFamilies} Let $f: (X, Z) \to B$ be a flat projective
  family with $\O_Z$ (and hence $\sI_Z$) flat over $B$ as well.  Assume that all the fiber pairs $(X_b, Z_b)$ are \DB.  Assume also that $B$ is
  connected and the generic fibers $(\sI_Z)_{\gen}$ are Cohen-Macaulay, then all the
  fibers $(\sI_Z)_b$ are Cohen-Macaulay.
\end{corollary}

At this point it is natural to ask the next question.

\begin{question}
  \label{quest.CMOutsideOfDB}
  Assume that $(X, Z)$ is a pair and that $H \subseteq X$ is a Cartier divisor such
  that $(H, Z\cap H)$ is a \DB pair.  If $\sI_Z|_{X \setminus H}$ is Cohen-Macaulay
  then is it true that $\sI_Z$ is Cohen-Macaulay?
\end{question}

In the case that $Z = \emptyset$, the analogous result holds in characteristic $p >
0$ for $F$-injective singularities by \cite[Appendix by K.~Schwede and
A.~K.~Singh]{HoriuchiMillerShimomoto}.

\section{A result of Kov\'acs-Schwede-Smith for pairs}

The goal of this section is to prove the analog of the main result of
\cite{KovacsSchwedeSmithLCImpliesDuBois} for pairs $(X, Z)$.

\begin{lemma}
  \label{lem.PairEqualsPushdown} Let $X$ be a normal $d$-dimensional variety, $Z
  \subsetneq X$ a reduced closed subscheme and $\Sigma\subsetneq X$ a codimension
  $\geq 2$ subset containing the singular locus of $X$.  Let $\pi : \tld X \to X$ be
  a log resolution of $(X, \Sigma \cup Z)$ with $E = \pi^{-1}(\Sigma \cup Z)_{\red}$.
  Then
  \begin{enumerate}
  \item $ \DuBois{X,\Sigma \cup Z}\simeq \myR \pi_* \O_{\tld X}(-E)$, and
  \item $ \myH^{-d}(\DBDual{X, Z}) \simeq \pi_* \omega_{\tld X}(E).$
  \end{enumerate}
\end{lemma}
\begin{proof}
  First we claim that both $\myR \pi_* \O_{\tld X}(-E)$ and $\pi_* \omega_{\tld
    X}(E)$ are independent of the choice of $\pi$. This was proved for $\myR \pi_*
  \O_{\tld X}(-E)$ on pages 67-68 in the proof of \cite[Theorem
  6.4]{KovacsSchwedeDuBoisSurvey} and for $\pi_* \omega_{\tld X}(E)$ in \cite[Lemma
  3.12]{KovacsSchwedeSmithLCImpliesDuBois}.  Therefore we are free to choose $\pi$
  and hence we may assume that it is an isomorphism outside of $\Sigma \cup Z$.  We
  have the distinguished triangle
  \[
  \xymatrix{%
    \DuBois{X} \ar[r] \ar[d]_{\simeq} & \myR \pi_* \DuBois{\tld X} \oplus \DuBois{
      \Sigma \cup Z} \ar[d]_{\simeq} \ar[r] & \myR \pi_* \DuBois{E} \ar[d]_{\simeq}
    \ar[r]^-{+1} & \\ \DuBois{X} \ar[r] & \myR \pi_* \O_{\tld X} \oplus
    \DuBois{\Sigma \cup Z} \ar[r] & \myR \pi_* \O_E \ar[r]^-{+1} & }
  \]
  The isomorphisms follow since $\tld X$ and $E$ are \DB.  In the next diagram the
  first two rows are distinguished triangles by definition cf.\ \autoref{eq:1}. The
  third row is simply the pushforward of a natural short exact sequence from $\tld
  X$. The previous diagram and \cite[Lemma 2.1]{KollarKovacsLCImpliesDB} (or simply the octahedral axiom) implies that
  $\alpha$ below is an isomorphism.  The other two isomorphisms again follow since
  $\tld X$ and $E$ are \DB. Note that the columns are \emph{not} exact.
  \[
  \xymatrix{%
    \DuBois{X,\Sigma \cup Z} \ar[r] \ar[d]_{\simeq}^\alpha & \DuBois{X} \ar[r] \ar[d]
    & \DuBois{\Sigma \cup Z} \ar[d] \ar[r]^-{+1} &\\ \myR \pi_* \DuBois{\tld X,E}
    \ar[r] \ar@{.>}[d] & \myR \pi_* \DuBois{\tld X} \ar[r] \ar[d]_{\simeq} & \myR
    \pi_* \DuBois{E} \ar[r]^-{+1}\ar[d]_{\simeq} &\\ \myR \pi_* \O_{\tld X}(-E)
    \ar[r] & \myR \pi_* \O_{\tld X} \ar[r] & \myR \pi_* \O_{E} \ar[r]^-{+1} & }
  \]
  It follows that the dotted arrow and hence its composition with $\alpha$ are also
  isomorphisms.  This proves (a).

  In order to prove (b), consider the map $\DuBois{X, \Sigma \cup Z} \to \DuBois{X,
    Z}$ obtained in
  \[ \xymatrix{ \DuBois{X, \Sigma \cup Z} \ar@{..>}[d] \ar[r] & \DuBois{X}
    \ar@{=}[d]\ar[r] & \DuBois{\Sigma \cup Z} \ar^-{+1}[r] \ar[d]& \\ \DuBois{X, Z}
    \ar[r] & \DuBois{X} \ar[r] & \DuBois{Z} \ar^-{+1}[r] &.  }
  \]
  Now we have a distinguished triangle
  \begin{equation}
    \label{eq.InducedTriangle}
    \DuBois{X, \Sigma \cup Z} \to \DuBois{X, Z} \to C^{\mydot} \xrightarrow{+1}.
  \end{equation}
  \begin{claim}
    \label{claim.ChasNoD}
    With notation above, $0 = \myH^{-d}(\myD(C^{\mydot})) =
    \myH^{-d+1}(\myD(C^{\mydot}))$.
  \end{claim}
  \begin{proof}[Proof of claim]
    Consider the following diagram with distinguished triangles as rows and columns:
    \[
    \xymatrix{%
      \DuBois{\Sigma \cup Z}[-1] \ar[d] \ar[r] & \DuBois{Z}[-1] \ar[d] \ar[r] &
      \DuBois{\Sigma \cup Z, Z}  \ar[r]^-{+1} &\\
      \DuBois{X, \Sigma \cup Z} \ar[d] \ar[r] & \DuBois{X, Z} \ar[d] \ar[r]
      &C^{\mydot} \ar[r]^-{+1} &\\ \DuBois{X} \ar[d]_-{+1} & \DuBois{X} \ar[d]_-{+1}
      & & \\ & & & }
    \]
    It follows from \cite[Theorem B.1]{Kovacs11b} that $C^{\mydot} \qis \DuBois{\Sigma \cup
      Z, Z}$.
    On the other hand, by \autoref{lem.BasicPropertiesOfDuBois}\autoref{item:4} $\DuBois{\Sigma
      \cup Z, Z}\simeq \DuBois{\Sigma, \Sigma \cap Z}$ and hence $C^{\mydot}\simeq
    \DuBois{\Sigma, \Sigma \cap Z}$.

    Next recall that by \autoref{thm.MainInjectivityForPairs} there exists a natural
    injective map
    \begin{equation}
      \label{eq:2}
      \myH^{-j} \left(\myD (\DuBois{\Sigma, \Sigma \cap Z})\right) \into \myH^{-j} \left(\myR
        \sHom^{\mydot}_{\O_{\Sigma}}(\sI_{(\Sigma \cap Z) \subseteq \Sigma},
        \omega_{\Sigma}^{\mydot})\right).
    \end{equation}
    Since $\dim \Sigma \leq d-2$, the right hand side of (\ref{eq:2}) is zero for $j
    \geq d-1$.  This completes the proof of the claim.
  \end{proof}
  Grothendieck duality and part (a) implies that $\myH^{-d}(\DBDual{X, \Sigma \cup
    Z}) \simeq \pi_*\omega_{\tld X}(E)$ and it follows from \autoref{claim.ChasNoD}
  that $\myH^{-d}(\DBDual{X, \Sigma \cup Z}) \simeq \myH^{-d}(\DBDual{X, Z})$, which
  in turn implies part (b).
\end{proof}

\begin{theorem}
  \label{thm.KSSPairs} Let $X$ be a normal variety and $Z \subseteq X$ a divisor.
  Let $\pi : \tld X \to X$ be a log resolution of $(X, Z)$ with $E =
  \pi^{-1}(Z)_{\red} \vee \exc(\pi)$.  If $\sI_Z$ is Cohen-Macaulay then $(X, Z)$ is
  Du Bois if and only if
  \[
  \pi_* \omega_{\tld X}(E) \simeq \omega_X(Z).
  \]
\end{theorem}
\begin{proof}
  Since $\sI_Z$ is Cohen-Macaulay, $\myR \sHom_{\O_X}^{\mydot}(\sI_Z,
  \omega_X^{\mydot}) \qis \sHom_{\O_X}(\sI_Z, \omega_X)[\dim X]$ by the local dual of
  the local cohomology criterion for Cohen-Macaulayness.  Because the map
  \begin{equation}
    \label{eq.DBDualToMapInjection}
    \DBDual{X, Z} \to \myR  \sHom_{\O_X}^{\mydot}(\sI_Z, \omega_X^{\mydot})
  \end{equation}
  is injective on cohomology by \autoref{thm.MainInjectivityForPairs}, it follows
  that $\myH^{i}(\DBDual{X,Z}) = 0$ for $i \neq -\dim X$ and hence $(X, Z)$ is Du
  Bois if and only if $\myH^{-\dim X}(\DBDual{X,Z}) \to \sHom_{\O_X}(\sI_Z, \omega_X)
  \simeq \omega_X(Z)$ is an isomorphism.  But $\myH^{-\dim X}(\DBDual{X,Z}) \simeq
  \pi_* \omega_{\tld X}(E)$ by \autoref{lem.PairEqualsPushdown}, so the statement
  follows.
\end{proof}

\section{An inversion of adjunction for rational and Du Bois pairs}

In this final section of the paper, we will prove the following theorem.

\begin{theorem}
  \label{thm.InvOfAdjForDuBoisPairs}
  Let $f : X \to B$ be a flat projective family with geometrically integral fibers over a smooth connected base $B$, $A
  \subseteq B$ a smooth closed subscheme containing no component of $B$ and $H =
  f^{-1} A = X \times_B A$ (with the induced scheme-theoretic structure).  Let $D$ be a reduced
  codimension $1$ subscheme of $X$ which is flat over $B$.
  %
  %
  Assume that for every $s\in A$, $(X_s, D_s)$ is \DB and that $(X \setminus H, D
  \setminus H)$ is a rational pair.  Then $(X, D)$ is a rational pair.
\end{theorem}

\begin{remark} In the introduction, $A$ was assumed to be a closed point.  This
  version is more general and more convenient for our proof.
\end{remark}

\begin{remark}
  \label{rem.additional.conditions}
  The assumptions also imply the following auxiliary conditions:
  \begin{enumerate}
  \item Since $X \to B$ has geometrically integral fibers and $H$ is obtained by base change with a smooth subscheme, $H$ is reduced.
  \item $\sI_D$ is flat over $B$ and no component of $D$ contains a fiber of $f$.
    In particular $D$ and $H$ have no common components.
  \item As for any $s\in A$, $H_s=X_s$, it follows that $(H, D\cap H)$ is \DB by
    \autoref{cor.DBDeformsArbitraryBase}.
  \item $X \setminus H$ is normal by the definition of a rational pair.
  \end{enumerate}
\end{remark}

Before embarking on proving the theorem, we will first prove several lemmas that show
that our situation is simpler than it might first appear.

First we show that we may assume that $A$ is a divisor in $B$.

\begin{lemma}
  \label{lem.BaseCanBeCurve} In order to prove \autoref{thm.InvOfAdjForDuBoisPairs}
  it is sufficient to assume that $A$ is a smooth Cartier divisor in $B$.
\end{lemma}
\begin{proof}
  The statement is local over the base so we may assume that $B$ is affine.
  Additionally, since we only need to work in a neighborhood of a point $a \in A$, we
  may assume that $(X, D)$ is \DB and all the fibers $(X_b, D_b)$ for all $b \in B$
  are \DB by \autoref{cor.DBDeformsArbitraryBase} and
  \autoref{cor.DBDeformsArbitraryBaseFibers}.  Choose a general hypersurface $G$
  containing $A$ and note that since $A$ is smooth we may assume that $G$ is smooth.
  Then the hypotheses of the theorem are satisfied for $G$ replacing $A$ as well
  since $X \setminus f^{-1}(G) \subseteq X \setminus f^{-1} (A)$ and since we already
  assumed that all the fibers $(X_b, D_b)$ over all points $b\in B$ were \DB.
\end{proof}

From this point forward, we will assume that $B$ is a smooth affine scheme, $A$ is a
smooth hypersurface in $B$ and $H = f^* A$.

Next we show that under the assumptions of the theorem we have the following:

\begin{lemma}
  \label{lemma.XNormal} $X$ is normal and thus $D$ is also a divisor.
\end{lemma}

\begin{proof}
  Since $H$ is reduced, every point $\eta\in H$ has depth at least $\min(1, \dim
  \O_{H,\eta})$.  Because $f : X \to B$ is flat, the local defining equation of $H$
  is a regular element in $\O_X$, so any point $\eta\in X$ that lies in $H$ has depth
  at least $\min(2, \dim \O_{X,\eta})$.  Since $X
  \setminus H$ is normal it is \stwo{} and so $X$ is \stwo{} everywhere.  Finally
  observe that $H$ is reduced, hence generically regular and $X \setminus H$ is
  \rone{}. As $H$ is Cartier, this implies that $X$ is also \rone{} and therefore
  normal.
\end{proof}

Now observe that the fact that $(H, D\cap H)$ is \DB (cf.\
\autoref{rem.additional.conditions}(b)) says something about the structure of $D$ on
$X$.

\begin{lemma}
\label{lem.NoStratumOfXDinH}
  With notation as in \autoref{thm.InvOfAdjForDuBoisPairs}, no stratum of the snc
  locus of $(X, D)$ can be contained inside $H$.
\end{lemma}
\begin{proof}
  Assume to the contrary that there exists a stratum $Z$ of the snc locus of $(X, D)$
  contained in $H$.  Let $\eta$ be the generic point of $Z$. By assumption $\eta\in
  H$ and $(X,D)$ is snc at $\eta$, so $\O_{X, \eta}$ is a regular ring.  Let $n =
  \dim \O_{X,\eta}$.  Replace $X$ by $\Spec \O_{X, \eta}$ and $H$ and $D$ by their
  pullbacks to this local scheme (in this step we lose projectivity, but we will not
  need that for now).  Note that $D$ is now Cartier and in fact snc.  Furthermore $D
  + H$ has $n+1$ irreducible components containing $\eta$, so $(X, D+H)$ cannot be
  Du Bois (or equivalently log canonical since $X$ is Gorenstein).  But as we observed, $(H, D\cap H)$ is a \DB pair.  
  Then by \autoref{thm.MainTheoremDeformation} again we see that $(X, D+H)$ is \DB as well.  This is a contradiction.
\end{proof}

Next we setup the notation for the proof of \autoref{thm.InvOfAdjForDuBoisPairs}.
Let $\Sigma$ denote the non-snc locus of $(X, D)$.
Observe that as $X$ is normal and $D$ is a reduced divisor by
\autoref{lemma.XNormal}, we have that $\codim_X(\Sigma) \geq 2$.

Additionally assume that $\pi : Y \to X$ is a log resolution
of $(X, D \cup H \cup \Sigma)$ which simultaneously gives a thrifty resolution of
$(X, D)$.  To see such a $\pi$ exists, first take a thrifty resolution $(U, D_U)$ of
$(X, D)$ and then perform a log resolution of the scheme-theoretic preimages of $H$
and $\Sigma$ on $U$ (while keeping the strict transform $D_U$ snc).  The result
can be assumed to be a thrifty resolution of $(X, D)$ since the preimages of $\Sigma$
and $H$ do not contain any strata of $(U, D_U)$ by \autoref{lem.NoStratumOfXDinH}.

Set $\overline{H}$ to be the reduced total transform of $H$ and $\overline{D}$ the
reduced total transform of $D$, set $D_Y$ to be the strict transform of $D$ and set
$E$ to be $\big(\pi^{-1}(\Sigma)\big)_{\red}$.

\begin{proof}[Proof of \autoref{thm.InvOfAdjForDuBoisPairs}]
  Clearly $(X, D)$ is a \DB pair and all the fibers $(X_b, D_b)$ are \DB by
  \autoref{cor.DBDeformsArbitraryBase} and \autoref{cor.DBDeformsArbitraryBaseFibers}
  (possibly after shrinking the base $B$ around $A$).  By
  \autoref{cor.CMClosedDuBoisFamilies}, we know that $\O_X(-D)$ is Cohen-Macaulay.
  Thus by the local dual version of the local-cohomological criterion for
  Cohen-Macaulayness, $\myR \sHom_{\O_X}^{\mydot}(\O_X(-D), \omega_X^{\mydot})$ has
  cohomology only in one term.  In particular,
  \begin{equation}
    \label{eq.AppOfPairKollarKovacs}
    \myR \sHom_{\O_X}^{\mydot}(\O_X(-D),
    \omega_X^{\mydot}) \qis \sHom_{\O_X}(\O_X(-D), \omega_X)[\dim X] \simeq
    \omega_X(D)[\dim X].
  \end{equation}
  Therefore by \autoref{cor:rational-pairs-criterion} it suffices to show that
  $\omega_X(D) \simeq \pi_* \omega_Y(D_Y)$.

  Next observe that $(H, D|_H)$ is a Du~Bois pair by \autoref{cor.DBDeformsArbitraryBase} and hence by
  \autoref{lem.TwoOfThreeDuBoisForPairs} we see that $(X, D \cup H) = (X, D+H)$ is a
  \DB pair.

  \begin{claim}
    \label{clm.PiPushforwardWithoutE}
    With notation as above, $\pi_* \omega_{Y}(D_Y \vee \overline{H} \vee E) \simeq
    \pi_* \omega_{Y}(D_Y + \overline{H} )$.
  \end{claim}

  \noindent
  Note that $D_Y + \overline{H} = D_Y \vee \overline{H}$ since the divisors have no
  common components.

  \begin{proof}[Proof of \autoref{clm.PiPushforwardWithoutE}]
    The containment $\supseteq$ is obvious since $D$ and $H$ do not share a component
    (cf.\ \autoref{rem.additional.conditions}(a)) so choose $f \in \pi_*
    \omega_{Y}(D_Y \vee \overline{H} \vee E)$.  We observe that
    \[
    \Div_Y(f) + K_Y + D_Y \vee \overline{H} \vee E = \Div_Y(f) + K_Y + D_Y + \overline{H} \vee E \geq
    0.
    \]
    Working on $U = Y \setminus \overline{H} = \pi^{-1}(X \setminus H)$ we see
    that $\Div_U(f) + K_{U} + D_Y|_U + E|_U \geq 0$.  But since $(X\setminus H, D
    \setminus H)$ is a rational pair,
    \[
    \pi_* \omega_{U}(D_Y|_U) = \pi_* \omega_{U}(D_Y|_U + E) = \omega_{X \setminus
      H}(D|_{X \setminus H}),
    \]
    so $\Div_U(f) + K_{U} + D_Y|_U + E|_U \geq 0$ is equivalent to $\Div_U(f) + K_U+
    D_Y|_U \geq 0$.  Because the components of $E$ that lie over $H$ are also
    components of $\overline{H}$, it follows that $\Div_{Y}(f)+ K_{Y} + D_Y +
    \overline{H} \geq 0$.  This proves \autoref{clm.PiPushforwardWithoutE}.
  \end{proof}

  By \autoref{lem.PairEqualsPushdown} we see that $\myH^{-\dim X} (\DBDual{X, D+H})
  \simeq \pi_* \omega_Y(D_Y \vee \overline{H} \vee E)$ which agrees with $\pi_*
  \omega_Y(D_Y + \overline{H})$ by the claim.  Since $(X, D+H)$ is a Du~Bois pair,
  $\myH^{-\dim X} (\DBDual{X, D+H}) \simeq \omega_X(D+H)$ and so in conclusion we have
  that
  \[
  \omega_X(D+H) \simeq \pi_* \omega_Y(D_Y + \overline{H}).
  \]
  Twisting both sides by $-H$ and using the projection formula we see that
  \[
  \omega_X(D) \simeq \pi_* \omega_Y(D_Y - (\pi^* H - \overline{H})) \subseteq \pi_*
  \omega_Y(D_Y)
  \]
  since $\pi^* H - \overline{H}$ is effective.  But $\pi_* \omega_Y(D_Y) \subseteq
  \omega_X(D)$ for any normal pair $(X, D)$ and so $\omega_X(D) \simeq \pi_*
  \omega_Y(D_Y)$ as desired.
\end{proof}

\noindent
Setting $D = 0$ we obtain the following:

\begin{corollary}
  \label{cor.DBRationalDeform}
  Let $f : X \to B$ be a flat projective family over a smooth base $B$ and $H =
  f^{-1}(0)$ a special fiber.  Assume that $H$ has Du Bois singularities and that $X
  \setminus H$ has rational singularities.  Then $X$ has rational singularities.
\end{corollary}

There is a variant of our inversion of adjunction theorem that we would also like to
prove (even in the $D = 0$ case).
\begin{conjecture}
  Assume $(X, D)$ is a pair with $D$ a reduced Weil divisor.  Further assume that $H$
  is a Cartier divisor on $X$ not having any components in common with $D$ such that
  $(H, D\cap H)$ is \DB and such that $(X \setminus H, D\setminus H)$ is a rational
  pair.  Then $(X, D)$ is a rational pair.
\end{conjecture}
The only place where our proof above does not work in this situation is when we prove
that $\O_X(-D)$ is Cohen-Macaulay.  In particular, to accomplish this generalization,
we would simply need a version of \autoref{cor.CMClosedDuBoisFamilies} that is not
tied to a projective or proper family. What is missing is exactly a positive answer
to \autoref{quest.CMOutsideOfDB}.

\bibliographystyle{skalpha}
\bibliography{MainBib}

\def\cfudot#1{\ifmmode\setbox7\hbox{$\accent"5E#1$}\else
  \setbox7\hbox{\accent"5E#1}\penalty 10000\relax\fi\raise 1\ht7
  \hbox{\raise.1ex\hbox to 1\wd7{\hss.\hss}}\penalty 10000 \hskip-1\wd7\penalty
  10000\box7}
\providecommand{\bysame}{\leavevmode\hbox to3em{\hrulefill}\thinspace}
\providecommand{\MR}{\relax\ifhmode\unskip\space\fi MR}
\providecommand{\MRhref}[2]{%
  \href{http://www.ams.org/mathscinet-getitem?mr=#1}{#2}
}
\providecommand{\href}[2]{#2}
\begin{thebibliography}{GKKP11}

\bibitem[AK80]{AltmanKleimanCompactifying}
{\sc A.~B. Altman and S.~L. Kleiman}: \emph{Compactifying the {P}icard scheme},
  Adv. in Math. \textbf{35} (1980), no.~1, 50--112. {\sf\scriptsize 555258
  (81f:14025a)}

\bibitem[DB81]{DuBoisMain}
{\sc P.~Du~Bois}: \emph{Complexe de de {R}ham filtr\'e d'une vari\'et\'e
  singuli\`ere}, Bull. Soc. Math. France \textbf{109} (1981), no.~1, 41--81.
  {\sf\scriptsize MR613848 (82j:14006)}

\bibitem[DJ74]{DuBoisJarraudAPropertyOfCommutation}
{\sc P.~Dubois and P.~Jarraud}: \emph{Une propri\'et\'e de commutation au
  changement de base des images directes sup\'erieures du faisceau structural},
  C. R. Acad. Sci. Paris S\'er. A \textbf{279} (1974), 745--747.
  {\sf\scriptsize 0376678 (51 \#12853)}

\bibitem[Elk78]{ElkikDeformationsOfRational}
{\sc R.~Elkik}: \emph{Singularit\'es rationnelles et d\'eformations}, Invent.
  Math. \textbf{47} (1978), no.~2, 139--147. {\sf\scriptsize MR501926
  (80c:14004)}

\bibitem[Eri14]{Lindsay14}
{\sc L.~Erickson}: \emph{Deformation invariance of rational pairs},
  arXiv:1407.0110 [math.AG].

\bibitem[Esn90]{EsnaultHodgeTypeOfSubvarieties}
{\sc H.~Esnault}: \emph{Hodge type of subvarieties of {${\bf P}^n$} of small
  degrees}, Math. Ann. \textbf{288} (1990), no.~3, 549--551. {\sf\scriptsize
  1079878 (91m:14075)}

\bibitem[FW89]{FedderWatanabe}
{\sc R.~Fedder and K.~Watanabe}: \emph{A characterization of {$F$}-regularity
  in terms of {$F$}-purity}, Commutative algebra (Berkeley, CA, 1987), Math.
  Sci. Res. Inst. Publ., vol.~15, Springer, New York, 1989, pp.~227--245.
  {\sf\scriptsize MR1015520 (91k:13009)}

\bibitem[GK14]{GK13}
{\sc P.~Graf and S.~J. Kov{\'a}cs}: \emph{Potentially {D}u {B}ois spaces}, J.
  Singul. \textbf{8} (2014), 117--134. {\sf\scriptsize 3395242}

\bibitem[GR70]{GRVanishing}
{\sc H.~Grauert and O.~Riemenschneider}: \emph{Verschwindungss\"atze f\"ur
  analytische {K}ohomologiegruppen auf komplexen {R}\"aumen}, Invent. Math.
  \textbf{11} (1970), 263--292. {\sf\scriptsize MR0302938 (46 \#2081)}

\bibitem[GKKP11]{MR2854859}
{\sc D.~Greb, S.~Kebekus, S.~J. Kov{\'a}cs, and T.~Peternell}:
  \emph{Differential forms on log canonical spaces}, Publ. Math. Inst. Hautes
  \'Etudes Sci. (2011), no.~114, 87--169. {\sf\scriptsize MR2854859}

\bibitem[Gro63]{EGAIII2}
{\sc A.~Grothendieck}: \emph{\'{E}l\'ements de g\'eom\'etrie alg\'ebrique.
  {III}. \'{E}tude cohomologique des faisceaux coh\'erents. {II}}, Inst. Hautes
  \'Etudes Sci. Publ. Math. (1963), no.~17, 91. {\sf\scriptsize MR0163911 (29
  \#1210)}

\bibitem[Hac14]{HaconLCInversionAdjunction}
{\sc C.~D. Hacon}: \emph{On the log canonical inversion of adjunction}, Proc.
  Edinb. Math. Soc. (2) \textbf{57} (2014), no.~1, 139--143. {\sf\scriptsize
  3165017}

\bibitem[HX13]{MR3032329}
{\sc C.~D. Hacon and C.~Xu}: \emph{Existence of log canonical closures},
  Invent. Math. \textbf{192} (2013), no.~1, 161--195. {\sf\scriptsize 3032329}

\bibitem[Har66]{HartshorneResidues}
{\sc R.~Hartshorne}: \emph{Residues and duality}, Lecture notes of a seminar on
  the work of A. Grothendieck, given at Harvard 1963/64. With an appendix by P.
  Deligne. Lecture Notes in Mathematics, No. 20, Springer-Verlag, Berlin, 1966.
  {\sf\scriptsize MR0222093 (36 \#5145)}

\bibitem[Har77]{Hartshorne}
{\sc R.~Hartshorne}: \emph{Algebraic geometry}, Springer-Verlag, New York,
  1977, Graduate Texts in Mathematics, No. 52. {\sf\scriptsize MR0463157 (57
  \#3116)}

\bibitem[HMS14]{HoriuchiMillerShimomoto}
{\sc J.~Horiuchi, L.~E. Miller, and K.~Shimomoto}: \emph{Deformation of
  {$F$}-injectivity and local cohomology}, Indiana Univ. Math. J. \textbf{63}
  (2014), no.~4, 1139--1157. {\sf\scriptsize 3263925}

\bibitem[Kar00]{KaruBoundedness}
{\sc K.~Karu}: \emph{Minimal models and boundedness of stable varieties}, J.
  Algebraic Geom. \textbf{9} (2000), no.~1, 93--109. {\sf\scriptsize MR1713521
  (2001g:14059)}

\bibitem[Kaw07]{KawakitaInversion}
{\sc M.~Kawakita}: \emph{Inversion of adjunction on log canonicity}, Invent.
  Math. \textbf{167} (2007), no.~1, 129--133. {\sf\scriptsize MR2264806
  (2008a:14025)}

\bibitem[KSB88]{KollarShepherdBarron}
{\sc J.~Koll{\'a}r and N.~I. Shepherd-Barron}: \emph{Threefolds and
  deformations of surface singularities}, Invent. Math. \textbf{91} (1988),
  no.~2, 299--338. {\sf\scriptsize MR922803 (88m:14022)}

\bibitem[Kol97]{KollarSingularitiesOfPairs}
{\sc J.~Koll{\'a}r}: \emph{Singularities of pairs}, Algebraic geometry---Santa
  Cruz 1995, Proc. Sympos. Pure Math., vol.~62, Amer. Math. Soc., Providence,
  RI, 1997, pp.~221--287. {\sf\scriptsize MR1492525 (99m:14033)}

\bibitem[Kol13]{KollarKovacsSingularitiesBook}
{\sc J.~Koll{\'a}r}: \emph{Singularities of the minimal model program},
  Cambridge Tracts in Mathematics, vol. 200, Cambridge University Press,
  Cambridge, 2013, With the collaboration of S{\'a}ndor Kov{\'a}cs.
  {\sf\scriptsize 3057950}

\bibitem[KK10]{KollarKovacsLCImpliesDB}
{\sc J.~Koll{\'a}r and S.~J. Kov{\'a}cs}: \emph{Log canonical singularities are
  {D}u {B}ois}, J. Amer. Math. Soc. \textbf{23} (2010), no.~3, 791--813.
  {\sf\scriptsize 2629988}

\bibitem[KM98]{KollarMori}
{\sc J.~Koll{\'a}r and S.~Mori}: \emph{Birational geometry of algebraic
  varieties}, Cambridge Tracts in Mathematics, vol. 134, Cambridge University
  Press, Cambridge, 1998, With the collaboration of C. H. Clemens and A. Corti,
  Translated from the 1998 Japanese original. {\sf\scriptsize MR1658959
  (2000b:14018)}

\bibitem[Kov99]{KovacsDuBoisLC1}
{\sc S.~J. Kov{\'a}cs}: \emph{Rational, log canonical, {D}u {B}ois
  singularities: on the conjectures of {K}oll\'ar and {S}teenbrink}, Compositio
  Math. \textbf{118} (1999), no.~2, 123--133. {\sf\scriptsize MR1713307
  (2001g:14022)}

\bibitem[Kov00]{KovacsDuBoisLC2}
{\sc S.~J. Kov{\'a}cs}: \emph{Rational, log canonical, {D}u {B}ois
  singularities. {II}. {K}odaira vanishing and small deformations}, Compositio
  Math. \textbf{121} (2000), no.~3, 297--304. {\sf\scriptsize MR1761628
  (2001m:14028)}

\bibitem[Kov11a]{KovacsDBPairsAndVanishing}
{\sc S.~J. Kov{\'a}cs}: \emph{Du {B}ois pairs and vanishing theorems}, Kyoto J.
  Math. \textbf{51} (2011), no.~1, 47--69. {\sf\scriptsize 2784747}

\bibitem[Kov11b]{MR2918171}
{\sc S.~J. Kov{\'a}cs}: \emph{Irrational centers}, Pure Appl. Math. Q.
  \textbf{7} (2011), no.~4, Special Issue: In memory of Eckart Viehweg,
  1495--1515. {\sf\scriptsize 2918171}

\bibitem[Kov13]{Kovacs11b}
{\sc S.~J. Kov{\'a}cs}: \emph{Singularities of stable varieties}, Handbook of
  Moduli, Volume II., Advanced Lectures in Mathematics, vol.~25, International
  Press, Somerville, MA, 2013, pp.~159--203.

\bibitem[KS11a]{KovacsSchwedeDBDeforms}
{\sc S.~J. Kov\'acs and K.~Schwede}: \emph{{Du Bois} singularities deform},
  arXiv:1107.2349, to appear in Advanced Studies in Pure Mathematics.

\bibitem[KS11b]{KovacsSchwedeDuBoisSurvey}
{\sc S.~J. Kov\'acs and K.~Schwede}: \emph{Hodge theory meets the minimal model
  program: a survey of log canonical and {D}u {B}ois singularities}, Topology
  of Stratified Spaces (G.~Friedman, E.~Hunsicker, A.~Libgober, and L.~Maxim,
  eds.), Math. Sci. Res. Inst. Publ., vol.~58, Cambridge Univ. Press,
  Cambridge, 2011, pp.~51--94.

\bibitem[KSS10]{KovacsSchwedeSmithLCImpliesDuBois}
{\sc S.~J. Kov{\'a}cs, K.~Schwede, and K.~E. Smith}: \emph{The canonical sheaf
  of {D}u {B}ois singularities}, Adv. Math. \textbf{224} (2010), no.~4,
  1618--1640. {\sf\scriptsize 2646306}

\bibitem[Lun12]{MR2981713}
{\sc V.~A. Lunts}: \emph{Categorical resolutions, poset schemes, and {D}u
  {B}ois singularities}, Int. Math. Res. Not. IMRN (2012), no.~19, 4372--4420.
  {\sf\scriptsize 2981713}

\bibitem[Ma15]{MaFInjectiveBuchsbaum}
{\sc L.~Ma}: \emph{{$F$}-injectivity and {B}uchsbaum singularities}, Math. Ann.
  \textbf{362} (2015), no.~1-2, 25--42. {\sf\scriptsize 3343868}

\bibitem[Pat15]{PatakfalviSemiNegativityHodgeBundles}
{\sc {\relax Zs}.~Patakfalvi}: \emph{Semi-negativity of {H}odge bundles
  associated to {D}u {B}ois families}, J. Pure Appl. Algebra \textbf{219}
  (2015), no.~12, 5387--5393. {\sf\scriptsize 3390027}

\bibitem[Sai00]{SaitoMixedHodge}
{\sc M.~Saito}: \emph{Mixed {H}odge complexes on algebraic varieties}, Math.
  Ann. \textbf{316} (2000), no.~2, 283--331. {\sf\scriptsize MR1741272
  (2002h:14012)}

\bibitem[Sch12]{MR2969273}
{\sc G.~Schumacher}: \emph{Positivity of relative canonical bundles and
  applications}, Invent. Math. \textbf{190} (2012), no.~1, 1--56.
  {\sf\scriptsize 2969273}

\bibitem[Sch07]{SchwedeEasyCharacterization}
{\sc K.~Schwede}: \emph{A simple characterization of {D}u {B}ois
  singularities}, Compos. Math. \textbf{143} (2007), no.~4, 813--828.
  {\sf\scriptsize MR2339829}

\bibitem[Sch09]{SchwedeFAdjunction}
{\sc K.~Schwede}: \emph{{$F$}-adjunction}, Algebra Number Theory \textbf{3}
  (2009), no.~8, 907--950. {\sf\scriptsize 2587408 (2011b:14006)}

\bibitem[ST08]{SchwedeTakagiRationalPairs}
{\sc K.~Schwede and S.~Takagi}: \emph{Rational singularities associated to
  pairs}, Michigan Math. J. \textbf{57} (2008), 625--658.

\bibitem[Ser56]{MR0082175}
{\sc J.-P. Serre}: \emph{G\'eom\'etrie alg\'ebrique et g\'eom\'etrie
  analytique}, Ann. Inst. Fourier, Grenoble \textbf{6} (1955--1956), 1--42.
  {\sf\scriptsize 0082175 (18,511a)}

\end{thebibliography}
\end{document}